\documentclass[12pt,english]{article}
\usepackage{mathptmx}
\usepackage{berasans}
\usepackage{beramono}

\usepackage[T1]{fontenc}
\usepackage[latin9]{inputenc}
\usepackage{geometry}
\usepackage{color}
\usepackage{babel}
\usepackage{float}
\usepackage{amsthm}
\usepackage{amsmath}
\usepackage{amssymb}
\usepackage{graphicx}
\usepackage{esint}
\usepackage[authoryear]{natbib}

\makeatletter

  \theoremstyle{definition}
  
  \theoremstyle{plain}
  \newtheorem{prop}{\protect\propositionname}[section]
  \theoremstyle{definition}
  \newtheorem{example}{\protect\examplename}[section]
 \ifx\proof\undefined\
   \newenvironment{proof}[1][\proofname]{\par
     \normalfont\topsep6\p@\@plus6\p@\relax
     \trivlist
     \itemindent\parindent
     \item[\hskip\labelsep
           \scshape
       #1]\ignorespaces
   }{%
     \endtrivlist\@endpefalse
   }
   \providecommand{\proofname}{Proof}
 \fi
  \theoremstyle{plain}
  \newtheorem{lem}{\protect\lemmaname}[section]
  \theoremstyle{remark}
  \newtheorem{rem}{\protect\remarkname}[section]
  \theoremstyle{plain}
  \newtheorem{thm}{\protect\theoremname}[section]
  \theoremstyle{plain}
  \newtheorem{cor}{\protect\corollaryname}[section]

\usepackage{indentfirst}

\makeatother

\providecommand{\corollaryname}{Corollary}
\providecommand{\definitionname}{Definition}
\providecommand{\examplename}{Example}
\providecommand{\lemmaname}{Lemma}
\providecommand{\propositionname}{Proposition}
\providecommand{\remarkname}{Remark}
\providecommand{\theoremname}{Theorem}

\begin{document}

\title{\textbf{A family of transformed copulas with singular component}}

\author{\textit{Jiehua Xie$^{\rm a}$, \ \ Jingping Yang$^{\rm b}$, \ \ Wenhao Zhu$^{\rm a}$}
  \\
\small  \textit{$^{\rm a}$Department of Financial Mathematics, Peking University, Beijing 100871, P.R.China}\\ \small
\textit{$^{\rm b}$LMEQF, Department of Financial Mathematics, Peking University, Beijing 100871, P.R.China}}

\date{October 2017}
\maketitle
\begin{abstract}
In this paper, we present a family of bivariate copulas by transforming a given copula function with two increasing functions, named as transformed copula. One distinctive characteristic of the transformed copula is its singular component along the main diagonal. Conditions guaranteeing the transformed function to be a copula function are provided, and several classes of the transformed copulas are given. The singular component along the main diagonal of the transformed copula is verified, and the tail dependence coefficients of the transformed copulas are obtained.
Finally, some properties of the transformed copula are discussed, such as the totally positive of order 2 and the concordance order.
\\
\textbf{\textit{Keywords}}\textit{:} Transformed copula; Singular component; Tail dependence.
\end{abstract}

\section{Introduction}

A copula is
a joint distribution function with all uniform [0, 1] marginal distributions. Copula functions have been received a great deal of attentions
due to Sklar's Theorem, providing a description of every joint distribution function about a random vector by its marginal distribution functions and its copula function. For more detailed introduction about copula theory, we refer to \citet{Joe(1997)} and \citet{Nelsen(2006)}.
Nowadays, copula functions play an important role for modeling the dependence structure in insurance, finance, risk management and econometrics \citep{Denuit et al.(2005),McNeil et al.(2005)}.

As a distribution function, a copula can be written as the sum of an absolutely continuous component and a singular component \citep[see, e.g.,][Theorem 2.2.6]{Ash(2000)}. Many popular copula functions, such as the Gaussian copula, the student's $t$ copula and the Farlie-Gumbel-Morgenstern (FGM) copula, do not have singular component. However, the increasing importance and widespread applications of copulas have required to model the dependence structure by using copula functions with singular component. In the credit risk modelling, for modeling the dependence structure between the default times of the two credit entities, the copula function
serves to describe the event that two entities default simultaneously, please refer to \citet{Sun et al. (2010)}, \citet{Mai and Scherer(2014)} and \citet{Bo and Capponi(2015)}. In the engineering applications, for modeling the lifetimes of two components in the same system which is influenced by exogenous shocks, the corresponding lifetimes take the same value if a shock hits two components simultaneously \citep{Navarro and Spizzichino(2010),Navarro et al.(2013)}. The copula functions with the singular component along the main diagonal should be employed to model the
dependence structure in the above applications.

Construction of copula functions with singular component is important for practical applications. The Marshall-Olkin copula is originated from exogenous shock models \citep{Marshall and Olkin(1967)} having singular component. \cite{Durante et al.(2007)} presented a generalization of the Archimedean family of bivariate copulas with a singular component along the main diagonal. \cite{Mai et al.(2016)} provided a family of the copula functions interrelated with the set of exchangeable exogenous shock models, and the copula functions have singular components along the main diagonal. \cite{Xie et al.(2017)} introduced a new family of the multivariate copula functions defined by two generators, which is a generalization of the multivariate Archimedean copula family with a singular component along the main diagonal.

In this paper, we consider a transformed function $C_{\phi,\psi}$ of a given copula $C$ with two functions $\phi$ and $\psi$ as the following:
\begin{equation}\label{c1}C_{\phi,\psi}(u,v)=\phi^{[-1]}(C(\phi(u\wedge v),\psi(u\vee v))),\ u,v\in [0,1],\end{equation}
where $u\wedge v=\min(u,v)$ and $u\vee v=\max(u,v)$.
Our motivation is to construct new families of copulas by applying both the functions $\phi,\psi$ and the copula $C$.
Conditions on the functions $\phi,\psi$ and the copula $C$ are provided to guarantee that the transformation $C_{\phi,\psi}$ is a copula function, and the copula function $C_{\phi,\psi}$ is named as the transformed (TF) copula and $C$ as the base copula. We will show that the TF copula has the following properties:

(1). Singularization. If the base copula is absolutely continuous, the TF copula may have a singular component along the main diagonal to take into account the event that two entities default simultaneously.

(2). Modification of the tail dependence. The tail dependence of the base copula can be changed by choosing two different functions $\phi$ and $\psi$.

(3). Preservation of some orders and properties. The TF copula preserves the common concordance order and the totally positive of order 2 (TP$_{2}$) property of the base copula under some conditions.

(4). Inclusion of many known copula families. The family of TF copulas includes many known bivariate copulas, such as the copula functions presented in \cite{Durante et al.(2007)}, \cite{Mai et al.(2016)} and \cite{Xie et al.(2017)}. The TF copula family can be regarded as a generalization of the above copula functions.

The rest of the paper is organized as follows. In Section 2, we provide sufficient conditions guaranteeing that $C_{\phi,\psi}$ is a copula function, and the existence of a singular component of the TF copula is verified. Several classes of TF copulas are also provided in Section 2. In Section 3, we discuss the tail dependence of the copula function $C_{\phi,\psi}$, and the tail dependence relationship between the copula $C$ and the TF copula is provided. In Section 4, some properties of the TF copula, such as the TP$_{2}$ and the concordance order, are discussed.
Conclusions are drawn in Section 5. Some proofs are put in Appendix.

\section{The family of transformed copulas}

In this section, we first give some notations that will be useful for introducing the function $C_{\phi,\psi}$ defined by \eqref{c1}. The sufficient conditions are provided such that $C_{\phi,\psi}$ is a copula function, and a singular component of the TF copula is verified. Finally, several classes of TF copulas are provided.

\subsection{Preliminaries}

As a multivariate distribution, the copula function has all uniform [0,1] marginal distributions. In the two-dimensional case, there are three important copula functions, the product copula $\Pi(u,v)=uv,u,v\in [0,1]$, the Fr\'{e}chet upper bound $M(u,v)=\min\{u,v\},u,v\in [0,1]$ and the Fr\'{e}chet lower bound  $W(u,v)=\max\{u+v-1, 0\},u,v\in [0,1]$.
It is known that for each bivariate copula $C$,
$$W(u,v) \leq C(u,v)\leq M(u,v),\ (u,v)\in [0,1]^{2}.$$
For details on copulas, please refer to \citet{Nelsen(2006)}.

Next we introduce the notion of supermigrative copula presented in \cite{Durante and Ricci(2012)}.
A bivariate copula $C$ is called supermigrative if it is exchangeable and satisfies that
\begin{equation}\label{d3}C(\alpha x, y)\geq C(x, \alpha y) \end{equation}
for all $\alpha\in[0,1]$ and $x, y\in[0,1]$ with $y\leq x$.
If the inequality in (\ref{d3}) is strict whenever $0<y <x\leq 1$, $C$ is called strictly supermigrative.

The supermigrative was introduced firstly in the
study of the bivariate ageing by \cite{Bassan and Spizzichino(2005)}. Supermigrative copulas are important for constructing stochastic models for the bivariate ageing \citep{Durante and Ricci(2012)}. Some common copulas are supermigrative. For instance, the product copula $\Pi$ and the Fr\'{e}chet upper bound $M$ are supermigrative. The Cuadras-Aug\'{e} bivariate copula ({\color{blue} Cuadras and Aug\'{e}}, {\color{blue} 1981})
\begin{equation}\label{CACOPULA}
C(u,v)=\left\{\begin{array}{ll}
uv^{\alpha},& \ u\leq v,\\
u^{\alpha}v,&  \ u> v,
\end{array}\right.
\end{equation}
with parameter $\alpha\in [0,1]$ is supermigrative. Furthermore, some popular copulas are also supermigrative under certain conditions. An Archimedean copula $C_{\varphi}(u,v)$ with generator $\varphi$
is defined as
\begin{equation} \label{eq:dis} C(u,v)=\varphi^{[-1]}(\varphi(u)+\varphi(v)),\end{equation}
where the generator $\varphi:$ $[0,1]\rightarrow$ $[0,+\infty]$ is a continuous and strictly decreasing convex function  satisfying $\varphi(1)=0$.
The Archimedean copula
is supermigrative if and only if the inverse function of $\varphi$ is log-convex and $\varphi(0)=+\infty$.
A bivariate FGM copula
\begin{equation}\label{FG} C(u,v)=uv(1+\theta (1-u)(1-v))
\end{equation}
is supermigrative if and only if the parameter $\theta\in[0,1]$. A bivariate Gaussian copula $C_{\vartheta}$ with correlation index parameter $\vartheta$ is supermigrative if and only if $\vartheta\geq 0$.

Let $\phi:$ $[0,1]\rightarrow$ $[0,1]$ be a continuous and strictly increasing function with $\phi(1)=1$. When $\phi(0)=0$, the function $\phi$ is a distortion function \citep[see, e.g.,][]{Yaari(1987)}, which is widely used in risk management. The pseudo-inverse of $\phi$ is defined by
$$\phi^{[-1]}(t):=\left\{\begin{array}{ll}
0,& t\in[0,\phi(0)],\\
\phi^{-1}(t),& t\in[\phi(0),1].
\end{array}
\right.$$
It is easy to see that the function $\phi^{[-1]}$ is increasing in [0,1] and strictly increasing in $[\phi(0),1]$. Specially, if $\phi(0)=0$, then the pseudo-inverse function coincides with the inverse function, that is, $\phi^{[-1]}=\phi^{-1}$.

From the definition of the pseudo-inverse $\phi^{[-1]}$, we can see that for all $t\in[0,1]$,
\begin{equation}\label{eq:pphi}\phi^{[-1]}(\phi(t))=t,\end{equation}
and
\begin{equation}\label{eq:pphi1}\phi\left(\phi^{[-1]}(t)\right)=\max\{t,\phi(0)\}.\end{equation}

Finally, we introduce some notations. Denote
$$\Psi:=\{\psi: [0,1]\rightarrow [0,1]\ |\ \psi {\rm\ is\ continuous,\ increasing\ and\ satisfies}\ \psi(1)=1 \},$$
and
$$\Phi:=\{\phi: [0,1]\rightarrow [0,1]\ |\ \phi {\rm\ is\ continuous,\ strictly\ increasing\ and\ satisfies}\ \phi(1)=1 \}.$$
Notice that $\Phi\subset\Psi$.

\subsection{Definition of the TF copula}

Given a base copula $C$ and a pair of functions $(\phi,\psi)\in\Phi\times\Psi$, the transformed function $C_{\phi,\psi}:$ $[0,1]^{2}\rightarrow$ $[0,1]$ is defined by (\ref{c1}).
 From the definition of $C_{\phi,\psi}$, we can see that the copula $C$ and the two functions $\phi$ and $\psi$ are three essential elements of the function $C_{\phi,\psi}$. In the next, we will prove that under some assumptions, the function $C_{\phi,\psi}(u,v)$ is a copula function.

First we give some preliminary lemmas.
\begin{lem}
\label{lem: MO} \citep[Proposition 4.B.2]{Marshall and Olkin(1979)} Let
 $f:A\rightarrow \mathbb{R}$, where $A$ is an interval of $\mathbb{R}$. If $f$ is convex and increasing, then, for each $a_{1}$, $a_{2}$, $a_{3}$, $a_{4}\in A$ such that
$$a_{1}\leq \min(a_{2}, a_{3})\leq \max(a_{2}, a_{3})\leq a_{4},$$
and $a_{1}+a_{4}\geq a_{2}+a_{3}$, we have
$$f(a_{1})+f(a_{4})\geq f(a_{2})+f(a_{3}).$$
 \end{lem}

 \begin{lem}\label{lem: convex} \citep[Page 290]{Stewart(2003)} Let $\phi\in \Phi$. If $\phi$ is a concave function, then $\phi^{[-1]}$ is convex in [0,1].
 \end{lem}

In the following, we will provide sufficient conditions on the functions $\phi$, $\psi$ and the copula $C$ such that the function $C_{\phi,\psi}$ of type (\ref{c1}) is a bivariate copula.
\begin{thm}
\label{thm:Distorted copula} Let $C$ be an arbitrary bivariate copula. Assume that $(\phi,\psi)\in\Phi\times\Psi$ and $\phi$ is a concave function. If the inequality
\begin{equation}\label{d1}C(\phi(u),\psi(v))\leq C(\phi(v),\psi(u))\end{equation}
holds for all $u\leq v, (u,v)\in[0,1]^{2}$, then the function $C_{\phi,\psi}(u,v)$ of type (\ref{c1}) is a bivariate copula. Specially, if $C$ is supermigrative and $\phi/\psi$ is increasing in [0,1], then $C_{\phi,\psi}$ is a copula function.
\end{thm}
\begin{proof}
(1). First we prove that when \eqref{d1} holds, $C_{\phi,\psi}$ is a copula function.

For each $u\in[0,1]$,
$$C_{\phi,\psi}(u,1)=\phi^{[-1]}(C(\phi(u),\psi(1)))=\phi^{[-1]}(\phi(u))=u,$$
and
$$0\leq C_{\phi,\psi}(u,0)=\phi^{[-1]}(C(\phi(0),\psi(u)))\leq \phi^{[-1]}(C(\phi(0),\psi(1)))=\phi^{[-1]}(\phi(0))=0.$$
Similarly, for each $v\in[0,1]$,
$$C_{\phi,\psi}(1,v)=v,\quad C_{\phi,\psi}(0,v)=0.$$

 Now, it suffices to prove that $C_{\phi,\psi}$ is 2-increasing, i.e., for every subset $R=[u_{1},u_{2}]\times[v_{1},v_{2}]$ in the unit square,
\begin{equation}V_{C_{\phi,\psi}}(R)=\phi^{[-1]}(a_{1})+\phi^{[-1]}(a_{4})-\phi^{[-1]}(a_{2})-\phi^{[-1]}(a_{3})\geq 0,\label{VCF}
\end{equation}
where
$$a_{1}=C(\phi(u_{1}\wedge v_{1}),\psi(u_{1}\vee v_{1})),\ \ a_{2}=C(\phi(u_{2}\wedge v_{1}),\psi(u_{2}\vee v_{1})),$$
$$a_{3}=C(\phi(u_{1}\wedge v_{2}),\psi(u_{1}\vee v_{2})),\ \ a_{4}=C(\phi(u_{2}\wedge v_{2}),\psi(u_{2}\vee v_{2})).$$
In order to show \eqref{VCF}, we will focus on the following three special cases which are illustrated in Figure \ref{fig:case}:
 \begin{figure}[htbp]
 \centering
\includegraphics[scale=0.586]{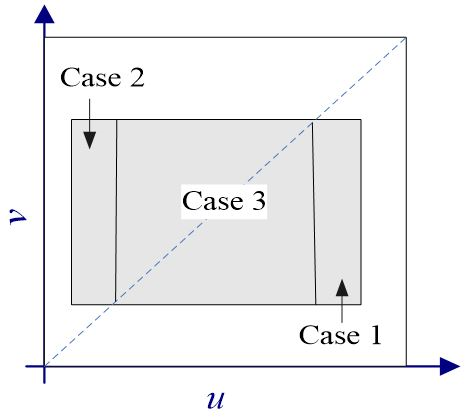}
\caption{Three special cases of the subsets $R=[u_{1},u_{2}]\times[v_{1},v_{2}]$ in the unit square.
\label{fig:case}}
\end{figure}
 \begin{itemize}
\item Case 1:  $R$ is a rectangle contained in the triangular region $G^{-}:=\{(u,v)\in[0,1]^{2}|u\geq v\}$ entirely.
\item Case 2: $R$ is a rectangle contained in the triangular region $G^{+}:=\{(u,v)\in[0,1]^{2}|u\leq v\}$ entirely.
\item Case 3: The diagonal of $R$ lies on the diagonal of the unit square.
\end{itemize}
Note that each rectangle $R$ in $[0,1]^{2}$ can be decomposed as the union of at most three non-intersect sub-rectangles $R_{i}$ of the above cases. Furthermore, $V_{C_{\phi,\psi}}(R)$ is the sum of the values $V_{C_{\phi,\psi}}(R_{i})$. Hence, we can verify that $C_{\phi,\psi}$ is 2-increasing if $C_{\phi,\psi}$ is 2-increasing in the above three cases.

In the next, we only need to prove that $C_{\phi,\psi}$ is 2-increasing in the above three cases. Since $C_{\phi,\psi}$ is exchangeable, then the arguments of the first two cases are similar. We only prove that $C_{\phi,\psi}$ is 2-increasing in Case 1 and Case 3.

\begin{itemize}
\item Case 1. Since $\phi$ and $\psi$ are increasing in [0,1], then $0\leq\phi(v_{1})\leq \phi(v_{2})\leq 1$ and $0\leq\psi(u_{1})\leq \psi(u_{2})\leq 1$ for $0\leq v_{1}\leq v_{2}\leq 1, 0\leq u_{1}\leq u_{2}\leq 1$. In this case,
   $$a_{1}=C(\phi(v_{1}),\psi(u_{1})),\ \ a_{2}=C(\phi(v_{1}),\psi(u_{2})),\ \ a_{3}=C(\phi(v_{2}),\psi(u_{1})),\ \ a_{4}=C(\phi(v_{2}),\psi(u_{2})).$$
Since the copula function $C$ is increasing in every variable, then $a_{1}\leq \min(a_{2}, a_{3})\leq \max(a_{2}, a_{3})\leq a_{4}$ follows. Furthermore, using the fact that $C$ is a bivariate copula and consequently $C$ is 2-increasing,
we have that $a_{1}+a_{4}\geq a_{2}+a_{3}$. Now, noting that $\phi$ is concave in [0,1], from Lemma \ref{lem: convex}, we know that $\phi^{[-1]}$ is convex in [0,1]. Then applying Lemma \ref{lem: MO} to the function $\phi^{[-1]}$, we can get \eqref{VCF}.

\item Case 3. In this case, $u_{1}=v_{1}$ and $u_{2}=v_{2}$. For $u_{1}\leq u_{2}$, $a_{1}=C(\phi(u_{1}),\psi(u_{1}))$,
$a_{2}=a_{3}=C(\phi(u_{1}),\psi(u_{2}))$ and $a_{4}=C(\phi(u_{2}),\psi(u_{2}))$. It is easy to see that $a_{1}\leq a_{2}=a_{3}\leq a_{4}$. Using the properties that $C$ is 2-increasing as well as $\phi$ and $\psi$ are increasing in [0,1], we have
\begin{equation}\label{d2}C(\phi(u_{2}),\psi(u_{2}))+C(\phi(u_{1}),\psi(u_{1}))-C(\phi(u_{1}),\psi(u_{2}))-C(\phi(u_{2}),\psi(u_{1}))\geq0 \end{equation}
for $0\leq u_{1}\leq u_{2}\leq1$. From the assumption (\ref{d1}) and the inequality (\ref{d2}), we have
\begin{eqnarray*}
a_{1}+a_{4}&=&C(\phi(u_{2}),\psi(u_{2}))+C(\phi(u_{1}),\psi(u_{1}))\\
&\geq& C(\phi(u_{1}),\psi(u_{2}))+C(\phi(u_{2}),\psi(u_{1}))\\
&\geq& C(\phi(u_{1}),\psi(u_{2}))+C(\phi(u_{1}),\psi(u_{2}))\\
&=&a_{2}+a_{3}.
\end{eqnarray*}
Since $\phi^{[-1]}$ is increasing and convex, Lemma \ref{lem: MO} ensures that $V_{C_{\phi,\psi}}(R)\geq0$.
\end{itemize}

\noindent (2). Now we assume that $C$ is a supermigrative copula and $\phi/\psi$ is increasing in [0,1]. In order to prove that $C_{\phi,\psi}$ is a bivariate copula, we only need to show that under the assumptions, $C(\phi(u),\psi(v))\leq C(\phi(v),\psi(u))$ holds for all $u\leq v$, $(u,v)\in[0,1]^{2}$.

Since $(\phi,\psi)\in\Phi\times\Psi$ and $\phi/\psi$ is increasing in [0,1], then for all $t\in[0,1]$, $\frac{\phi(t)}{\psi(t)}\leq\frac{\phi(1)}{\psi(1)}=1$, thus $\phi(t)\leq\psi(t)$ follows. Then for all $u\leq v$, $(u,v)\in[0,1]^{2}$,
$$C(\phi(v),\psi(u))=C\left(\frac{\phi(v)}{\psi(v)}\psi(v),\psi(u)\right)\geq C\left(\frac{\phi(u)}{\psi(u)}\psi(v),\psi(u)\right)$$
$$\ \ \ \ \ \ \ \ \ \ \ \ \ \ \ \ \ \ \ \ \ \ \ \ \ \ \ \ \ \ \ \ \ \ \ \ \ \ \ \ \geq C\left(\psi(v),\frac{\phi(u)}{\psi(u)}\psi(u)\right)=C(\psi(v),\phi(u))=C(\phi(u),\psi(v)),$$
where the first inequality holds because $\phi/\psi$ is increasing in [0,1] and $C$ is increasing in every variable, the second inequality follows from (\ref{d3}), and the last equation holds because $C$ is exchangeable by its supermigrative property. Thus \eqref{d1} holds and $C_{\phi,\psi}$ is a copula function.
\end{proof}

\begin{rem}\label{re:compare}One advantage of the TF copula is that the copula family includes many known bivariate copulas. Letting the base copula $C$ is exchangeable and $\phi=\psi$, then the TF copula $C_{\phi,\psi}$ coincides with the copula $C_{\phi}$ given by
 \begin{equation}
C_{\phi}(u,v)=\phi^{[-1]}(C(\phi(u),\phi(v))).\label{eq:dis+}
\end{equation}
Please see \citet{Genest and Rivest(2001)}, \citet{Klement et al.(2005)}, \citet{Morillas(2005)}, \citet{Alvoni et al. (2009)} and \citet{Durante et al.(2010)} for the detailed discussion on the copula $C_{\phi}$. The Archimedean copula with multiplicative generator is also a special case of (\ref{eq:dis+}) with $C=\Pi$.

The Cuadras-Aug\'{e} copula is a special TF copula where the base copula $C$ is the product copula and $\phi(u)=u, \psi(u)=u^{\alpha}, \alpha\in[0,1]$.
In particular, let the base copula $C$ be the product copula $\Pi$, then the TF-product copula is simplified as
 \begin{equation}\Pi_{\phi,\psi}(u,v)=\phi^{[-1]}(\phi(u\wedge v)\cdot\psi(u\vee v)).\label{eq:dis++}
\end{equation}
Note that the function $\Pi_{\phi,\psi}$ coincides with the generalization of Archimedean
copula introduced by \cite{Durante et al.(2007)}.


\end{rem}\

\begin{rem}\label{re:inequality} When $C(\phi(u),\psi(v))\leq C(\phi(v),\psi(u))$ for all $u\leq v, (u,v)\in[0,1]^{2}$, using $\phi(1)=\psi(1)=1$, for $u\in[0,1]$, we have
$$\phi(u)=C(\phi(u),\psi(1))\leq C(\phi(1),\psi(u))\leq C(1,\psi(u))=\psi(u).$$
\end{rem}\

Given a bivariate copula $C$, we denote
$$\mathfrak{S}(C)=\{(\phi,\psi)\in\Phi\times\Psi|\ C(\phi(u),\psi(v))\leq C(\phi(v),\psi(u))\ {\rm
holds\ for\ all}\ u\leq v,\ (u,v)\in[0,1]^{2}\}$$ and let $\mathfrak{F}(C)$ be the family of all functions $(\phi,\psi)\in\Phi\times\Psi$ such that
$C_{\phi,\psi}$ is a copula. From Theorem \ref{thm:Distorted copula}, we know that if $(\phi,\psi)\in \mathfrak{S}(C)$ and $\phi$ is concave,
     then $(\phi,\psi)\in \mathfrak{F}(C)$. In the following, some properties of $\mathfrak{S}(C)$ and the relationships between $\mathfrak{S}(C)$ and $\mathfrak{F}(C)$ are presented, which will allow us to find more generators $(\phi,\psi)$ from $\mathfrak{S}(C)$ such that $C_{\phi,\psi}$
is a copula.

\begin{prop}
\label{prop: Tp1} Let $C$ be a bivariate copula. The following statements hold:

(1). If $(\phi,\psi)\in \mathfrak{S}(C)$ and $f\in \Phi$, then $(\phi\circ f,\psi\circ f)\in \mathfrak{S}(C)$. Furthermore, if $\phi\circ f$ is concave in [0,1], then $(\phi\circ f,\psi\circ f)\in \mathfrak{F}(C)$.

(2). If $(\phi_{1},\psi)\in \mathfrak{S}(C)$ and $(\phi_{2},\psi)\in \mathfrak{S}(C)$, then $(\max\{\phi_{1}, \phi_{2}\},\psi)\in \mathfrak{S}(C)$. Furthermore, if $\max\{\phi_{1}, \phi_{2}\}$ is concave in [0,1], then $(\max\{\phi_{1}, \phi_{2}\},\psi)\in \mathfrak{F}(C)$.

(3). If $(\phi,\psi_{1})\in \mathfrak{S}(C)$ and $(\phi,\psi_{2})\in \mathfrak{S}(C)$, then $(\phi, \max\{\psi_{1}, \psi_{2}\})\in \mathfrak{S}(C)$. Furthermore, if $\phi$ is concave in [0,1], then $(\phi, \max\{\psi_{1}, \psi_{2}\})\in \mathfrak{F}(C)$.
\end{prop}

\begin{proof} (1). From the assumption that $f\in \Phi$, we know that $f$ is a strictly increasing function in [0,1] satisfying $f(1)=1$. Then $f(u)\leq f(v)$ for $0\leq u\leq v\leq 1$ and $(\phi\circ f, \psi\circ f)\in \Phi\times\Psi$. Since $(\phi,\psi)\in \mathfrak{S}(C)$, then for all $u\leq v$, $(u,v)\in[0,1]^{2}$, it holds that
$$C(\phi\circ f(u),\psi\circ f(v))=C(\phi(f(u)),\psi(f(v)))\leq C(\phi(f(v)),\psi(f(u)))=C(\phi\circ f(v),\psi\circ f(u)).$$
Thus $(\phi\circ f,\psi\circ f)\in \mathfrak{S}(C)$. Furthermore, if $\phi\circ f$ is concave in [0,1], then the assumptions of Theorem \ref{thm:Distorted copula} are satisfied such that $(\phi\circ f,\psi\circ f)\in \mathfrak{F}(C)$. \\
(2). From the assumptions that $(\phi_{1},\psi)\in \mathfrak{S}(C)$ and $(\phi_{2},\psi)\in \mathfrak{S}(C)$, we easily know that $\max\{\phi_{1}, \phi_{2}\}$ is a continuous and strictly increasing function in [0,1] satisfying $\max\{\phi_{1}(1), \phi_{2}(1)\}=1$ such that $(\max\{\phi_{1}, \phi_{2}\},\psi)\in\Phi\times\Psi$. From the assumptions, we also know that $C(\phi_{i}(u),\psi(v))\leq C(\phi_{i}(v),\psi(u))$ for all $u\leq v$, $(u,v)\in[0,1]^{2}$ and $i\in\{1,2\}$.
Then for all $u\leq v$, $(u,v)\in[0,1]^{2}$, it holds that
$$C(\max\{\phi_{1}(u), \phi_{2}(u)\},\psi(v))=\max\{C(\phi_{1}(u),\psi(v)), C(\phi_{2}(u),\psi(v))\}\ \ \ \ \ \ \ \ \ \ \ \ \ \ \ \ \ \ \ \ \ \ \ \ \ \ \ \ \ \ \ \ \ \ \ \ \ \ \ \ \ \ \ \ \ \ \ $$
$$\ \ \ \ \ \ \ \ \ \ \ \ \ \ \ \ \ \ \ \ \ \ \ \ \ \ \ \ \ \ \ \ \ \ \ \ \ \ \ \ \ \ \ \ \ \ \ \ \ \ \ \leq\max\{C(\phi_{1}(v),\psi(u)), C(\phi_{2}(v),\psi(u))\}=C(\max\{\phi_{1}(v), \phi_{2}(v)\},\psi(u)).$$
Hence $(\max\{\phi_{1}, \phi_{2}\},\psi)\in \mathfrak{S}(C)$. Furthermore, when $\max\{\phi_{1}, \phi_{2}\}$ is concave in [0,1], the assumptions of Theorem \ref{thm:Distorted copula} are satisfied, thus $(\max\{\phi_{1}, \phi_{2}\},\psi)\in \mathfrak{F}(C)$.

 (3). It can be proved similarly as in (2).
\end{proof}

\subsection{A singular component of the TF copula}

In this section, it is verified that the TF copula $C_{\phi,\psi}$ may have a singular component even if the base copula $C$ is absolutely continuous. It is interesting to note that, if $C$ is absolutely continuous, then the TF copula $C_{\phi,\psi}$ has both singular and absolutely continuous components under some conditions. A singular component of the TF copula can be identified as in the following theorem.
 \begin{thm}
\label{thm:singular} Let $C$ be a copula function with second order continuous derivatives, $(\phi,\psi)\in\mathfrak{S}(C)$,  $\phi$ and $\psi$ be  differentiable in [0,1] and $\phi$ be a concave function. Denote $$S_{1}=\{(u,v)\in[0,1]^{2}|C(\phi(u\wedge v),\psi(u\vee v))>\phi(0)\}$$
and
$$S_{2}=\{(u,u)|C_{1}(\phi(u),\psi(u))\phi'(u)>C_{2}(\phi(u),\psi(u))\psi'(u),\ u\in[0,1]\},$$
where $C_{1}(u,v)=\frac{\partial C(u,v)}{\partial u}$ and $C_{2}(u,v)=\frac{\partial C(u,v)}{\partial v}$. Let $S=S_{1}\cap S_{2}$.

\noindent (1). If $S\neq\emptyset$, then the singular set of $C_{\phi,\psi}$ contains $S$.

\noindent (2). If $\phi(0)=0$ and $S\neq\emptyset$, then the singular component of the copula $C_{\phi,\psi}$ is concentrated on the main diagonal $\{(u,u)| u\in [0,1]\}$. Moreover, for the random variables $U$ and $V$ with the joint distribution $C_{\phi,\psi}$, we have
\begin{equation}\label{scp}\mathbb{P}(U=V)=S_{C_{\phi,\psi}}(1,1)=1-A_{C_{\phi,\psi}}(1,1),\end{equation}
where $S_{C_{\phi,\psi}}$ and $A_{C_{\phi,\psi}}$ are the singular component and the absolutely continuous component of $C_{\phi,\psi}$ respectively, i.e.,
\begin{equation}\label{dscp} C_{\phi,\psi}(u,v)=S_{C_{\phi,\psi}}(u,v)+A_{C_{\phi,\psi}}(u,v),\ \ {\rm and}\ \ A_{C_{\phi,\psi}}(u,v)=\int_{0}^{u}\int_{0}^{v}\frac{\partial^{2} C_{\phi,\psi}(s,t)}{\partial s\partial t}{\rm d}t{\rm d}s. \end{equation} \end{thm}
\begin{proof} (1). From the definitions of the copula $C_{\phi,\psi}$ and the pseudo-inverse $\phi^{[-1]}$, we have that for all $(u,v)\not\in S_{1}$, $C_{\phi,\psi}(u,v)\equiv 0$.  Hence we only need to focus on the set $S_{1}$.

 Fix $(u,v)\in S_{1}$. If $u\leq v$, then $C_{\phi,\psi}(u,v)=\phi^{[-1]}(C(\phi(u),\psi(v)))$. Since $(u,v)\in S_{1}$, from Eq.(\ref{eq:pphi1}), we get
$$\phi\left(C_{\phi,\psi}(u,v)\right)=C\left(\phi(u),\psi(v)\right).$$
From the assumptions that the copula function $C$ has second order continuous derivatives and $\phi$ is differentiable, then
differentiating the above equation with respect to $u$ yields
$$\phi^{'}\left(C_{\phi,\psi}(u,v)\right)\frac{\partial C_{\phi,\psi}(u,v)}{\partial u}=\frac{\partial C(\phi(u),\psi(v))}{\partial u}
=C_{1}\left(\phi(u),\psi(v)\right)\phi^{'}(u),\ \ \ u<v.$$
Similarly, for $u>v$,
$$\phi^{'}\left(C_{\phi,\psi}(u,v)\right)\frac{\partial C_{\phi,\psi}(u,v)}{\partial u}=\frac{\partial C(\phi(v),\psi(u))}{\partial u}
=C_{2}\left(\phi(v),\psi(u)\right)\psi^{'}(u).$$
Note that $\phi'(C_{\phi,\psi}(u,v))\neq 0$ in $S_{1}$. Hence,
\begin{equation}\label{singular1}
     \frac{\partial C_{\phi,\psi}(u,v)}{\partial u}=
    \begin{cases}
         C_{1}\left(\phi(u),\psi(v)\right)\frac{\phi'(u)}{\phi'(C_{\phi,\psi}(u,v))},\quad u<v, \\
         C_{2}\left(\phi(v),\psi(u)\right)\frac{\psi'(u)}{\phi'(C_{\phi,\psi}(u,v))},\quad u>v.
    \end{cases}
\end{equation}
Notice also that $\frac{\partial C_{\phi,\psi}(u,v)}{\partial u}=\mathbb{P}(V\leq v|U=u)$. Therefore, for a fixed $u$,
$$\lim_{v\rightarrow u^{+}}\frac{\partial C_{\phi,\psi}(u,v)}{\partial u}\geq\lim_{v\rightarrow u^{-}}\frac{\partial C_{\phi,\psi}(u,v)}{\partial u},$$
and from (\ref{singular1}) we know that the equality holds only if $C_{1}(\phi(u),\psi(u))\phi'(u)=C_{2}(\phi(u),\psi(u))\psi'(u)$.

Summarizing the above facts, we know that if $S=S_{1}\cap S_{2}\neq\emptyset$, then $\frac{\partial C_{\phi,\psi}(u,v)}{\partial u}$ has jump discontinuity in $S$. Hence, in view of Theorem 1.1 of \citet{Joe(1997)}, we conclude that the singular set of the copula $C_{\phi,\psi}$ contains $S$.
\\
(2). If $\phi(0)=0$, then $\phi^{[-1]}=\phi^{-1}$. From (\ref{c1}) we get
$$
    C_{\phi,\psi}(u,v)=
    \begin{cases}
         \phi^{-1}\left(C(\phi(u),\psi(v))\right),\quad u<v, \\
         \phi^{-1}\left(C(\phi(v),\psi(u))\right),\quad u>v.
    \end{cases}
$$
From the assumptions that $(\phi,\psi)\in\mathfrak{S}(C)$ and $\phi$, $\psi$ are differentiable in [0,1], we know that $\phi^{-1}$ is differentiable and strictly increasing in [0,1]. Also, since the copula function $C$ has second order continuous derivatives, then we easily know that $C_{\phi,\psi}(u,v)$ is absolutely continuous in $\{(u,v)\in[0,1]^{2}|u\neq v\}$.

Note that $S$ is concentrated on the main diagonal. Thus, if $S\neq\emptyset$, then from the above analysis, we conclude that $C_{\phi,\psi}$ is absolutely continuous in $\{(u,v)\in[0,1]^{2}|u\neq v\}$ and its singular component is concentrated on the main diagonal.

Eq. (\ref{scp}) follows directly from Lebesgue decomposition Theorem \citep[see, e.g.,][Theorem 2.2.6]{Ash(2000)}.
\end{proof}

\begin{rem}
\label{REM1}
In Theorem \ref{thm:singular}, the condition $S\neq\emptyset$ is dependent on the base copula $C$. Note that the condition can be simplified if the base copula $C$ is exchangeable.
Suppose that $C$ is exchangeable and $C_{1}(u,v)=g(u,v)h(u)$, where $g(u,v)$ is also exchangeable and satisfies that $g(u,v)>0$ in $[0,1]^{2}$. Then when   $\int_{\psi(u)}^{\phi(u)}h(t)dt$ is strictly increasing in [0,1], we can prove that $S\neq\emptyset$ and $C_{\phi,\psi}$ contains a singular component along the main diagonal. The proof will be given in Appendix.

As an example, choosing an Archimedean copula with
additive generator $\varphi$ as the base copula, then $h(u)=-\varphi^{'}(u)$ and if the function $\varphi\circ\psi-\varphi\circ\phi$ is strictly increasing in $[0,1]$, the copula $C_{\phi,\psi}$ contains a singular component along the main diagonal. Furthermore, letting the base copula being the product copula and $\varphi(t)=-\ln t$,  the condition can be simplified as that
$\phi/\psi$ is strictly increasing in $[0,1]$.
\end{rem}

\subsection{Some classes of TF copulas}

Besides including many important families of bivariate copulas, the TF copula of type (\ref{c1}) is of importance since such transformation allows us to construct new families of copulas with singular components along the main diagonal.

In the following, we consider TF copulas by choosing the FGM copula, the Cuadras-Aug\'{e} copula and the Archimedean copula as the base copula respectively.
It is well-known that the FGM copula is absolutely continuous and it has a polynomial form, thus we choose the FGM copula as the first example. For comparison purpose, in the second example we choose the Cuadras-Aug\'{e} copula as the base copula, because it has a singular component along the main diagonal. Since the Archimedean copula is an important family of copulas in many areas, we take the Archimedean copula as the third example.

\begin{example}
\label{EX: FGM} (TF-FGM copula)
The FGM copula of type \eqref{FG}  is supermigrative when $\theta\in[0,1]$. If $(\phi,\psi)\in\Phi\times\Psi$ and $\phi$ is a concave function such that $\phi/\psi$ is increasing in [0,1], then from Theorem \ref{thm:Distorted copula}, we know that
\begin{equation}\label{FGM}C_{\phi,\psi}(u,v)=\phi^{[-1]}(\phi(u\wedge v)\psi(u\vee v)(1+\theta(1-\phi(u\wedge v))(1-\psi(u\vee v)))),\ \ \theta\in[0,1]
\end{equation}
is a bivariate copula.

Let $\phi(t)=t^{\beta}$ and $\psi(t)=t^{\gamma}$, where $0\leq\gamma\leq\beta\leq1$. In this case, $(\phi,\psi)\in\Phi\times\Psi$, $\phi$ is a concave function and $\phi/\psi$ is increasing in [0,1]. Then \eqref{FGM}
can be written as
\begin{equation}\label{FGMPower}C_{\phi,\psi}(u,v)=(u\wedge v)(u\vee v)^{\frac{\gamma}{\beta}}\left(1+\theta(1-(u\wedge v)^{\beta})(1-(u\vee v)^{\gamma})\right)^{\frac{1}{\beta}}, \ \ \theta\in[0,1]. \end{equation}
Setting $\beta=1$, $\gamma=1/2$ in \eqref{FGMPower}, then
$$S=\{(u,u)|C_{1}(\phi(u),\psi(u))\phi'(u)>C_{2}(\phi(u),\psi(u))\psi'(u),\ u\in[0,1]\}$$
$$=\{(u,u)|\sqrt{u}/2+\theta(\sqrt{u}/2-3/2u^{3/2}+u^{2})>0,\ u\in[0,1]\}. \ \ \ \ \ \ \ \ $$
Define $Q(u)=\sqrt{u}/2+\theta(\sqrt{u}/2-3/2u^{3/2}+u^{2})$, $u\in[0,1]$.
Since $Q^{'}(u)=1/(4\sqrt{u})+\theta(1+8u^{3/2}-9u)/(4\sqrt{u})>1/(4\sqrt{u})+\theta(1-u^{3/2})/(4\sqrt{u})>0$ for all $u\in(0,1]$ and $\theta\in[0,1]$, then $Q(u)$ is an increasing function in [0,1]. Furthermore, it is easy to see that $Q(0)=0$ and $Q(1)=1/2$. Thus we can get that $S=\{(u,u)| u\in(0,1]\}\neq \emptyset$ for all $\theta\in[0,1]$. From Theorem \ref{thm:singular}, we know that the TF-FGM copula contains a singular component along the main diagonal. Also, from the double integral in (\ref{dscp}) for the TF-FGM copula, we get
$$A_{C_{\phi,\psi}}(u,v)=u\sqrt{v}\left(1+\theta(1-u)(1-\sqrt{v})\right)-\frac{u^{3/2}(5+\theta(5-9u+5u^{3/2}))}{15}, \ u<v,$$
and
$$A_{C_{\phi,\psi}}(u,v)=v\sqrt{u}\left(1+\theta(1-v)(1-\sqrt{u})\right)-\frac{v^{3/2}(5+\theta(5-9v+5v^{3/2}))}{15}, \ u>v.$$
Hence, we have
$$A_{C_{\phi,\psi}}(u,v)=C_{\phi,\psi}(u,v)-\frac{(u\wedge v)^{3/2}(5+\theta(5-9(u\wedge v)+5(u\wedge v)^{3/2}))}{15},\ \ \theta\in[0,1],$$
and consequently
$$S_{C_{\phi,\psi}}(u,v)=\frac{(u\wedge v)^{3/2}(5+\theta(5-9(u\wedge v)+5(u\wedge v)^{3/2}))}{15},\ \ \theta\in[0,1].$$
Furthermore, $$\mathbb{P}(U=V)=S_{C_{\phi,\psi}}(1,1)=\frac{5+\theta}{15},\ \ \theta\in[0,1],$$
where $U$ and $V$ are the uniform [0,1] random variables following the joint distribution function (\ref{FGMPower}) with $\beta=1$ and $\gamma=1/2$.
\end{example}
\

 In the following example,  we choose the Cuadras-Aug\'{e} copula as the base copula. It is known that the Cuadras-Aug\'{e} copula also has a singular component along the main diagonal.
\begin{example}
\label{EX: CA} (TF-Cuadras-Aug\'{e} copula) Recall that the Cuadras-Aug\'{e} copula of type \eqref{CACOPULA}
   is supermigrative. Let $\phi(t)=t^{\beta}$ and $\psi(t)=t^{\beta}(2-t^{\gamma})$, $t\in[0,1]$, where $0\leq\gamma\leq\beta\leq1$. In this case, $(\phi,\psi)\in\Phi\times\Psi$, $\phi$ is a concave function and $\phi/\psi$ is increasing in [0,1]. Thus
\begin{equation}\label{capower}C_{\phi,\psi}(u,v)=(u\wedge v)(u\vee v)^{\alpha}(2-(u\vee v)^{\gamma})^{\frac{\alpha}{\beta}},\ \ \ 0\leq\alpha\leq 1 \end{equation}
is a copula function.
Setting $\beta=\alpha$ in (\ref{capower}), we have
$$A_{C_{\phi,\psi}}(u,v)=C_{\phi,\psi}(u,v)-(u\wedge v)^{1+\alpha}\left(\frac{4}{1+\alpha}-\frac{(u\wedge v)^{\gamma}(1-\alpha-\gamma)}{1+\alpha+\gamma}-2\right),\ \ 0\leq\alpha\leq 1,\ \ 0\leq\gamma\leq\alpha,$$
and consequently
$$S_{C_{\phi,\psi}}(u,v)=(u\wedge v)^{1+\alpha}\left(\frac{4}{1+\alpha}-\frac{(u\wedge v)^{\gamma}(1-\alpha-\gamma)}{1+\alpha+\gamma}-2\right),\ \ 0\leq\alpha\leq 1,\ \ 0\leq\gamma\leq\alpha.$$
For $(U_{1},V_{1})$ having the copula (\ref{capower}) with $\beta=\alpha$ and $(U_{2},V_{2})$ having the Cuadras-Aug\'{e} copula $C$, we have
$$\mathbb{P}(U_{1}=V_{1})=\frac{4}{1+\alpha}-\frac{1-\alpha-\gamma}{1+\alpha+\gamma}-2,\ \ 0\leq\alpha\leq 1,\ \ 0\leq\gamma\leq\alpha$$
and
$$\mathbb{P}(U_{2}=V_{2})=\frac{1-\alpha}{1+\alpha},\ \ 0\leq\alpha\leq 1.$$
Then
$$\mathbb{P}(U_{1}=V_{1})-\mathbb{P}(U_{2}=V_{2})=\frac{2\gamma}{(1+\alpha)(1+\alpha+\gamma)}\geq 0.$$
Hence, the value of $P(U_1=V_1)$ for the copula of type (\ref{capower}) with $\beta=\alpha$ is larger than $P(U_2=V_2)$ for the Cuadras-Aug\'{e} copula. It shows that the transformation (\ref{c1}) enlarges the singular component along the main diagonal of the Cuadras-Aug\'{e} copula.
\end{example}
\

In the next example, we choose the Archimedean copula as the base copula.
\begin{example}
\label{EX: AC} (TF-Archimedean copula) Let the base copula $C$ be an Archimedean copula of type \eqref{eq:dis} with  generator $\varphi$ and $(\phi,\psi)\in\Phi\times\Psi$.
Then
$$C_{\phi,\psi}(u,v)=\phi^{[-1]}\circ\varphi^{[-1]}(\varphi\circ\phi(u\wedge v)+\varphi\circ\psi(u\vee v)).$$
From Lemma 3.7 in \cite{Morillas(2005)}, we can know that $(\varphi\circ\phi)^{[-1]}=\phi^{[-1]}\circ\varphi^{[-1]}$. Then $C_{\phi,\psi}(u,v)$ can be simplified as
\begin{equation}\label{ex1}C_{\phi,\psi}(u,v)=(\varphi\circ\phi)^{[-1]}((\varphi\circ\phi)(u\wedge v)+(\varphi\circ\psi)(u\vee v)).\end{equation}
From Theorem \ref{thm:Distorted copula}, we know that when $\phi$ is concave and the function $\varphi\circ\phi-\varphi\circ\psi$ is decreasing in [0,1],
then the function of type (\ref{ex1}) is a copula.

In particular, let generator $\varphi(t)=1-t$, $t\in[0.1]$. In this case, the base copula $C$ is the Fr\'{e}chet lower
bound $W$. We choose the concave function $\phi\in\Phi$ and the function $\psi\in\Psi$ such that $\psi-\phi$ is decreasing in [0,1]. Then the assumption (\ref{d1}) holds and Theorem \ref{thm:Distorted copula} ensures that
$$W_{\phi,\psi}(u,v)=\phi^{[-1]}(\max\{\phi(u\wedge v)+\psi(u\vee v)-1,0\})$$
is a copula function.
\end{example}

\subsection{Numerical illustration}

For understanding the singular component of the TF
copula $C_{\phi,\psi}$, in the following we will give the scatter plots for some
TF-Archimedean copulas.

Letting the base copula $C$ be an Archimedean copula with generator $\varphi$, then the TF-Archimedean copula has simple form (\ref{ex1}). The Clayton copula, the Gumbel copula and the Frank copula are three important classes of copulas in the Archimedean family and their generators can be expressed as
$\varphi(t)=t^{-\alpha}-1$, $\varphi(t)=(-\ln t)^{\beta}$ and $\varphi(t)=-\ln\left(\frac{e^{-\gamma t}-1}{e^{-\gamma}-1}\right)$ respectively.
We choose three base copulas, which are the Clayton copula with parameter $2$, the Gumbel copula with parameter $3$ and the Frank copula with parameter $4$ for numerical illustrations.

In order to get a comparison result, we choose four pairs of functions $(\phi, \psi)$: (a) $(\phi(t), \psi(t))=(t^{\frac{4}{5}}, t^{\frac{1}{2}})$; (b) $(\phi(t), \psi(t))=(t^{\frac{2}{3}}, t^{\frac{1}{2}})$; (c) $(\phi(t), \psi(t))=(t^{\frac{1}{2}}, t^{\frac{1}{2}})$; (d) $(\phi(t), \psi(t))=(t, t)$. It is easy to verify that for every Archimedean generator $\varphi$ mentioned above,  $\phi$ is concave and $(\varphi\circ\phi-\varphi\circ\psi)$ is decreasing in all the four cases described above. Then the assumption (\ref{d1}) holds such that $C_{\phi,\psi}$ is a bivariate copula in each case. Note that for every TF-Archimedean copula described above, $\phi\neq\psi$ and $S\neq \emptyset$ in case (a) and case (b). Then, all the three classes of TF-Archimedean copulas contain a singular component along the main diagonal in
case (a) and case (b). Since $\phi=\psi$ in case (c), we know that the TF copula coincides with the copula $C_{\phi}$ of type (\ref{eq:dis+}) in this case. In case (d), $C_{\phi,\psi}=C$. Since the Archimedean copula $C$ is exchangeable, then $C_{1}(u,u)=C_{2}(u,u)$ such that $S=\emptyset$ in case (c) and case (d). Thus, all the three classes of TF-Archimedean copulas are absolutely continuous and should not contain a singular component along the main diagonal in both case (c) and case (d).

We generate 10,000 samples of the TF-Calyton copula, the TF-Gumbel copula and the TF-Frank copula and illustrate them in Figures \ref{fig:1}-\ref{fig:3}, respectively. We also calculate the values of Kendall's $\tau$ as well as Spearman's $\rho$ and the results are shown below.

($I$). The base copula is a Clayton copula with parameter $\alpha=2$: (a) $\tau$=0.6226, $\rho$=0.7712; (b) $\tau$=0.5220, $\rho$=0.6848; (c) $\tau$=0.3313, $\rho$=0.4756; (d) $\tau$=0.5023, $\rho$=0.6839.

($II$). The base copula is a Gumbel copula with parameter $\beta=3$: (a) $\tau$=0.8734, $\rho$=0.9608; (b) $\tau$=0.8017, $\rho$=0.9273; (c) $\tau$=0.6677, $\rho$=0.8499; (d) $\tau$=0.6575, $\rho$=0.8399.

($III$). The base copula is a Frank copula with parameter $\gamma=4$: (a) $\tau$=0.5768, $\rho$=0.7359; (b) $\tau$=0.4828, $\rho$=0.6469; (c) $\tau$=0.3297, $\rho$=0.4787; (d) $\tau$=0.3859, $\rho$=0.5548.

  \begin{figure}[htbp]
\includegraphics[scale=0.526]{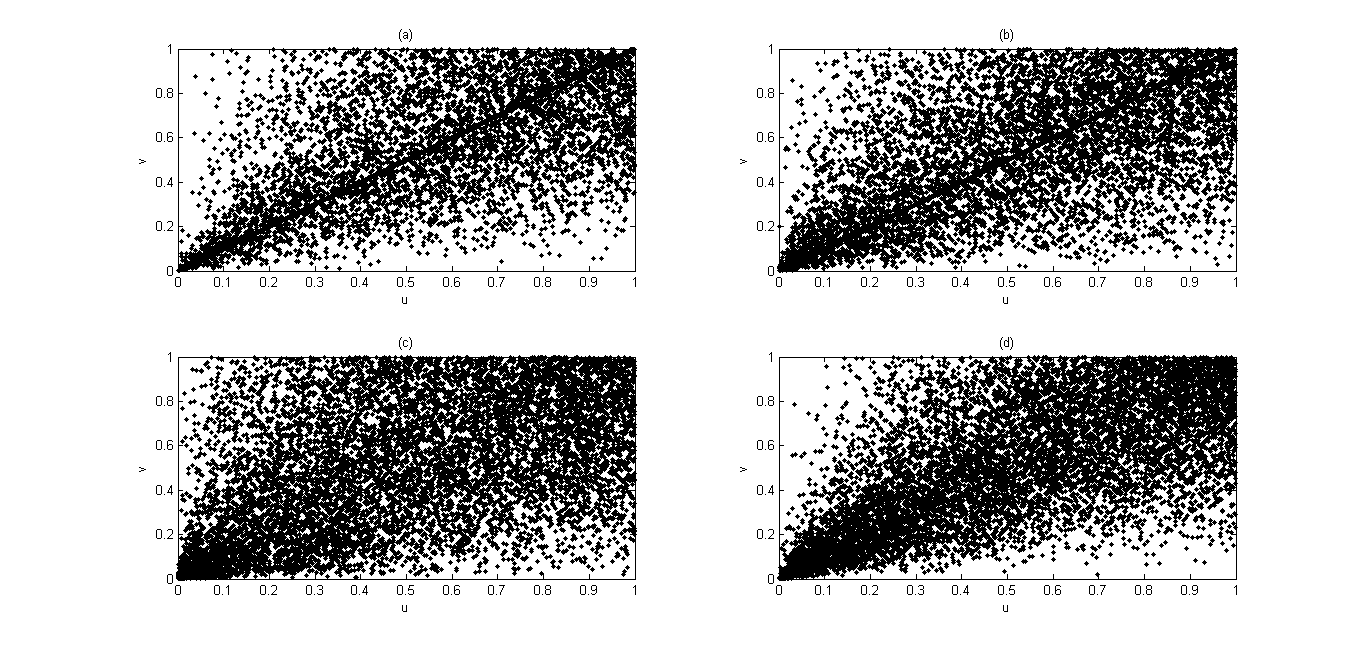}
\caption{The scatter plots for the TF-Calyton copula with parameter $\alpha=2$ for different pairs of functions: (a) $(\phi(t), \psi(t))=(t^{\frac{4}{5}}, t^{\frac{1}{2}})$; (b) $(\phi(t), \psi(t))=(t^{\frac{2}{3}}, t^{\frac{1}{2}})$; (c) $(\phi(t), \psi(t))=(t^{\frac{1}{2}}, t^{\frac{1}{2}})$; (d) $(\phi(t), \psi(t))=(t, t)$.
\label{fig:1}}
\end{figure}
\begin{figure}[htbp]
\includegraphics[scale=0.526]{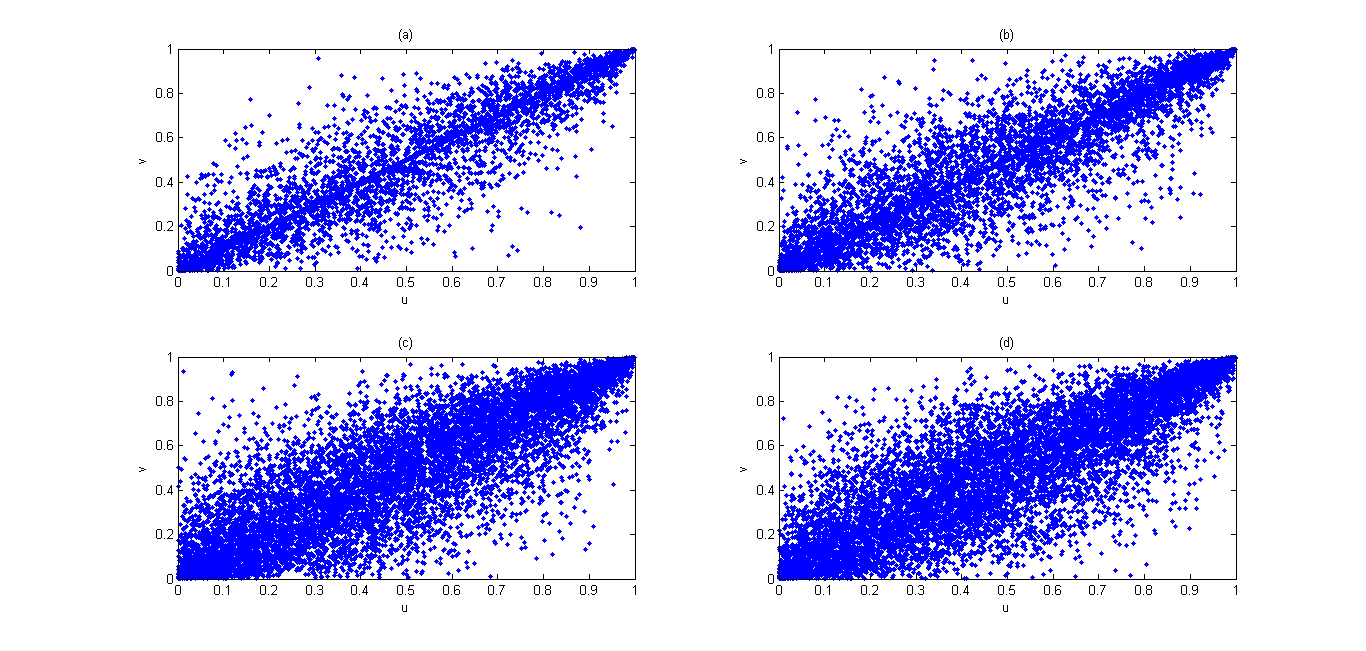}
\caption{The scatter plots for the TF-Gumbel copula with parameter $\beta=3$ for different pairs of functions: (a) $(\phi(t), \psi(t))=(t^{\frac{4}{5}}, t^{\frac{1}{2}})$; (b) $(\phi(t), \psi(t))=(t^{\frac{2}{3}}, t^{\frac{1}{2}})$; (c) $(\phi(t), \psi(t))=(t^{\frac{1}{2}}, t^{\frac{1}{2}})$; (d) $(\phi(t), \psi(t))=(t, t)$.
\label{fig:2}}
\end{figure}
\begin{figure}[htbp]
\includegraphics[scale=0.526]{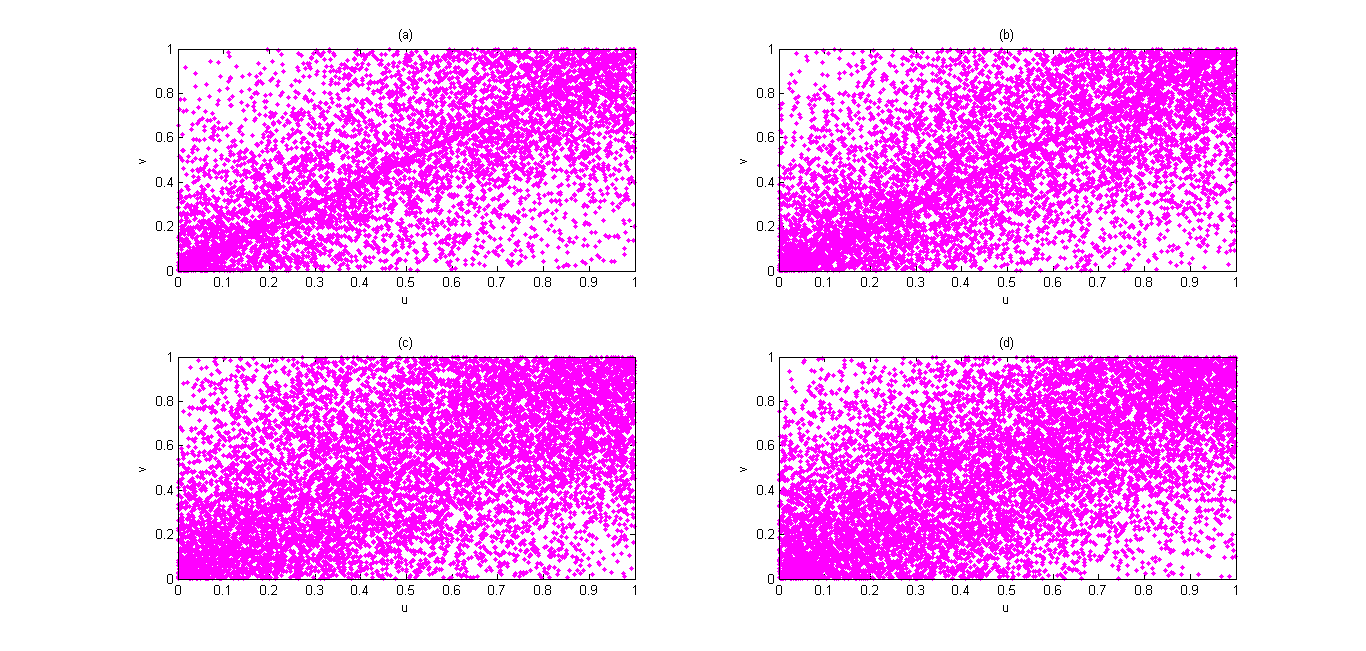}
\caption{The scatter plots for the TF-Frank copula with parameter $\gamma=4$ for different pairs of functions: (a) $(\phi(t), \psi(t))=(t^{\frac{4}{5}}, t^{\frac{1}{2}})$; (b) $(\phi(t), \psi(t))=(t^{\frac{2}{3}}, t^{\frac{1}{2}})$; (c) $(\phi(t), \psi(t))=(t^{\frac{1}{2}}, t^{\frac{1}{2}})$; (d) $(\phi(t), \psi(t))=(t, t)$.
\label{fig:3}}
\end{figure}

From the scatter plots presented in Figures \ref{fig:1}-\ref{fig:3}, we find that there are straight lines in both case (a) and case (b) of all the three pictures, which are the singular components along the main diagonal of the corresponding copulas. It can also be seen that
all cases (c) and (d) of the three base copulas do not contain a singular component. By Theorem \ref{thm:singular}, these results are conceivable since $S$ equals $\{(u,u)|0\le u\le 1\}$ for (a),(b) and $\emptyset$ for (c),(d). For every TF-Archimedean copula described above,
it is also observable that the values of Kendall's $\tau$ as well as Spearman's $\rho$ in case (a) and case (b) are greater than those in (c) and (d) when the base copula is the same. Thus the rank correlations are enhanced when the copula is transformed by (\ref{c1}).

\section{Tail dependence coefficient of the TF copula}

Copulas with different tail dependence are usually required for modelling the extreme events \citep{McNeil et al.(2005),Salvadori et al.(2007)}. Tail dependence coefficient was proposed for studying the tail dependence of copula functions \citep{Nelsen(2006),McNeil et al.(2005)}.
For a copula function $C$, its upper and lower tail dependence coefficients $\lambda_{U}(C)$ and $\lambda_{L}(C)$, respectively, are defined by
\begin{equation}\label{cam1}\lambda_{U}(C)=2-\lim_{u\rightarrow1^{-}}\frac{1-C(u,u)}{1-u},
\end{equation}
\begin{equation}\label{cam2}\lambda_{L}(C)=\lim_{u\rightarrow0^{+}}\frac{C(u,u)}{u},\ \ \ \ \ \ \ \ \ \ \ \ \end{equation}
 given that the above limits exist. The upper and lower tail dependence coefficients $\lambda_{U}(C)$ and $\lambda_{L}(C)$ have been shown to be of importance in the study of the tail dependence \citep{Nelsen(2006)}. We say that $C$ has upper tail dependence if
$0<\lambda_{U}(C)\leq1$, otherwise we say that $C$ has no upper tail dependence; similarly for $\lambda_{U}(L)$.

In the following proposition,  we give the tail dependence relationship between the base copula $C$ and its TF copula $C_{\phi,\psi}$.
\begin{prop}
\label{prop:utail} Let $C$ be a bivariate copula, $(\phi,\psi)\in\mathfrak{S}(C)$ and $\phi$ be a concave function.

\noindent (a). Suppose that $\lambda_{U}(C)$ exists and it is finite, and
for some $\alpha>0$,
$$\lim_{t\rightarrow1^{-}}\frac{1-\phi(t)}{(1-t)^{\alpha}}=a\in(0,+\infty),\ \ \lim_{t\rightarrow1^{-}}\frac{1-\psi(t)}{1-\phi(t)}=b.$$
\noindent (1). If $b=1$, then
\begin{equation}\label{up1}\lambda_{U}(C_{\phi,\psi})=2-\left(2-\lambda_{U}(C)\right)^{\frac{1}{\alpha}}.\end{equation}
\noindent (2). If $\lambda_{U}(C)=0$ and $ b\in[0,1]$, then
\begin{equation}\label{up2}\lambda_{U}(C_{\phi,\psi})=2-\left(1+b\right)^{\frac{1}{\alpha}}.
\end{equation}
(b). Suppose that $\lambda_{L}(C)$ exists and it is finite, and for some $\alpha>0$,
$$\lim_{t\rightarrow0^{+}}\frac{\phi(t)}{t^{\alpha}}=a\in(0,+\infty),\ \ \lim_{t\rightarrow0^{+}}\frac{\phi(t)}{\psi(t)}=b.$$
(1). If $b=1$, then
\begin{equation}\label{lo1}\lambda_{L}(C_{\phi,\psi})=\left(\lambda_{L}(C)\right)^{\frac{1}{\alpha}}.\end{equation}
(2). If $\lambda_{L}(C)=0$ and $b\in(0,1]$, then $\lambda_{L}(C_{\phi,\psi})=0$.
\end{prop}
\begin{proof}
(a).
Since $\phi$ is a continuous and strictly increasing function satisfying $\phi(1)=1$, then $C(\phi(u),\phi(u))<1, C(\phi(u),\psi(u))<1$ for $u\in [0,1)$. Thus, we have
 \begin{eqnarray}\label{1}
&\lambda_{U}(C_{\phi,\psi})&= 2-\lim_{u\rightarrow1^{-}}\frac{1-\phi^{[-1]}(C(\phi(u),\psi(u)))}{1-u}\nonumber\\
&&= 2-\lim_{u\rightarrow1^{-}}\left\{\left(\frac{1-\phi^{[-1]}(C(\phi(u),\psi(u)))}{(1-C(\phi(u),\psi(u)))^{\frac{1}{\alpha}}}
\cdot
\frac{(1-C(\phi(u),\phi(u)))^{\frac{1}{\alpha}}}{(1-\phi(u))^{\frac{1}{\alpha}}}\cdot\frac{(1-\phi(u))^{\frac{1}{\alpha}}}{1-u}\right)
\right. \nonumber\\
&&\ \ \ \ \ \ \ \ \ \ \ \  \ \ \ \ \ \ \ \ \ \ \left.\times
\left(\frac{1-C(\phi(u),\psi(u))}{1-C(\phi(u),\phi(u))}\right)^{\frac{1}{\alpha}}\right\},
\end{eqnarray}
 given that the above limit exists.
 In order to derive $\lambda_{U}(C_{\phi,\psi})$, we need to calculate $\beta:=\lim_{u\rightarrow1^{-}}[1-C(\phi(u),\psi(u))]/[1-C(\phi(u),\phi(u))]$ and
 $$\tilde{\beta}:=\lim_{u\rightarrow1^{-}}\left(\frac{1-\phi^{[-1]}(C(\phi(u),\psi(u)))}{(1-C(\phi(u),\psi(u)))^{\frac{1}{\alpha}}}
\cdot
\frac{(1-C(\phi(u),\phi(u)))^{\frac{1}{\alpha}}}{(1-\phi(u))^{\frac{1}{\alpha}}}\cdot\frac{(1-\phi(u))^{\frac{1}{\alpha}}}{1-u}\right)$$
 respectively if the limits exist.

 We first calculate the limit $\tilde{\beta}$. From (\ref{cam1}) and the assumptions that $\lambda_{U}(C)$ exists and it is finite, we know that \begin{equation}\label{up1+}\lim_{u\rightarrow1^{-}}\frac{1-C(u,u)}{1-u}=2-\lambda_{U}(C),\end{equation}
thus
\begin{equation}
\lim_{u\rightarrow1^{-}}\frac{(1-C(\phi(u),\phi(u)))^{\frac{1}{\alpha}}}{(1-\phi(u))^{\frac{1}{\alpha}}}
=(2-\lambda_{U}(C))^{\frac{1}{\alpha}}
\label{beta1}
\end{equation}
follows. On the other hand, from the assumption that $\lim_{u\rightarrow1^{-}}(1-\phi(u))/(1-u)^{\alpha}=a$, we know $\lim_{u\rightarrow1^{-}}$
$(1-u)/(1-\phi^{[-1]}(u))^{\alpha}$ $=a$. Since the functions $\phi$ and $\psi$ are continuous in [0,1] with $\phi(1)=\psi(1)=1$, we have $C(\phi(1), \psi(1))=1$. Thus we get
\begin{equation}
\lim_{u\rightarrow1^{-}}\left(\frac{1-\phi^{[-1]}(C(\phi(u),\psi(u)))}{(1-C(\phi(u),\psi(u)))^{\frac{1}{\alpha}}}
\cdot\frac{(1-\phi(u))^{\frac{1}{\alpha}}}{1-u}\right)=1.\label{beta2}
\end{equation}
Combining \eqref{beta1} and \eqref{beta2} we conclude that
\begin{eqnarray}
\label{up+}
\tilde{\beta}=(2-\lambda_{U}(C))^{\frac{1}{\alpha}}.
\end{eqnarray}

In the following, we will calculate the limit $\beta$
under two different conditions.
\\ \\
(1). We first consider the case $b=1$, i.e., $\lim_{u\rightarrow1^{-}}(1-\psi(u))/(1-\phi(u))=1$. In this case, we will apply Sandwich Theorem to calculate the limit $\beta$. Since $C$ is increasing in every variable and $\phi(u)\leq\psi(u)$ for all $u\in[0,1]$, we get
\begin{equation}\label{TFUinequality}C(\phi(u),\phi(u))\leq C(\phi(u),\psi(u))\leq C(\psi(u),\psi(u)),\end{equation}
such that
\begin{equation}\label{TFU1++}\frac{1-C(\psi(u),\psi(u))}{1-C(\phi(u),\phi(u))}\leq\frac{1-C(\phi(u),\psi(u))}{1-C(\phi(u),\phi(u))}\leq 1.\end{equation}
Now, we calculate the limit of the left side in the inequality (\ref{TFU1++}). Since $\psi(1)=\phi(1)=1$,
from (\ref{up1+}) we have
$$\lim_{u\rightarrow1^{-}}\left(\frac{1-C(\psi(u),\psi(u))}{1-\psi(u)}\cdot\frac{1-\phi(u)}{1-C(\phi(u),\phi(u))}\right)=1.$$
Based on the above result and the assumption that $\lim_{u\rightarrow1^{-}}(1-\psi(u))/(1-\phi(u))=1$, we derive
$$\lim_{u\rightarrow1^{-}}\frac{1-C(\psi(u),\psi(u))}{1-C(\phi(u),\phi(u))}
=\lim_{u\rightarrow1^{-}}\left\{\left(\frac{1-C(\psi(u),\psi(u))}{1-\psi(u)}\cdot\frac{1-\phi(u)}{1-C(\phi(u),\phi(u))}\right)\cdot\frac{1-\psi(u)}{1-\phi(u)}\right\}=1.$$
Then, letting $u\rightarrow1^{-}$ and applying Sandwich Theorem to the inequality (\ref{TFU1++}), we get $\beta=1$. Substituting this result and (\ref{up+}) into (\ref{1}), we can obtain (\ref{up1}).
\\ \\
(2). Now we consider the case $\lambda_{U}(C)=0$, $b\in[0,1]$. For a copula function $C$, we define \begin{equation}\label{TFU1}f(u_{1},u_{2})=1+C(u_{1},u_{2})-u_{1}-u_{2},\ (u_{1},u_{2})\in[0,1]^{2}.\end{equation}
Note that
\begin{equation}\label{EQTFU1}f(u_{1},u_{2})=C(1,1)+C(u_{1},u_{2})-C(u_{1},1)-C(1,u_{2})\geq 0.\end{equation}
Combining (\ref{up1+}) and (\ref{TFU1}), we have
$$2-\lambda_{U}(C)=2-\lim_{u\rightarrow1^{-}}
\frac{f(u,u)}{1-u}.$$
Thus
\begin{equation}\label{TFU2}\lim_{u\rightarrow1^{-}}
\frac{f(u,u)}{1-u}=\lambda_{U}(C)\end{equation}
follows.

Next we calculate the limit $\beta$. Note that
\begin{eqnarray}\label{TFU3}
&\beta&=\lim_{u\rightarrow1^{-}}\frac{1-C(\phi(u),\psi(u))}{1-C(\phi(u),\phi(u))}\nonumber\\
&&=\lim_{u\rightarrow1^{-}}\frac{1-\phi(u)+1-\psi(u)
-f(\phi(u),\psi(u))}
{1-\phi(u)+1-\phi(u)
-f(\phi(u),\phi(u))}\nonumber\\
&&=\lim_{u\rightarrow1^{-}}\frac{1+\frac{1-\psi(u)}{1-\phi(u)}
-\frac{f(\phi(u),\psi(u))}{1-\phi(u)}}
{2
-\frac{f(\phi(u),\phi(u))}{1-\phi(u)}},
\end{eqnarray}
where the second equality follows from the definition of $f(u_{1},u_{2})$ given in (\ref{TFU1}). Then, in order to get  the limit (\ref{TFU3}), we need to calculate $\lim_{u\rightarrow1^{-}}(1-\psi(u))/(1-\phi(u))$, $\lim_{u\rightarrow1^{-}}[f(\phi(u),\phi(u))/(1-\phi(u))]$ and $\lim_{u\rightarrow1^{-}}[f(\phi(u),\psi(u))/(1-\phi(u))]$ respectively if the limits exist.

From the assumptions, we know that $\lim_{u\rightarrow1^{-}}(1-\psi(u))/(1-\phi(u))=b$. Since $\phi(1)=1$ and $\lambda_{U}(C)=0$, from (\ref{TFU2}) we get
\begin{equation}\label{TFU1+}\lim_{u\rightarrow1^{-}}\frac{f(\phi(u),\phi(u))}{1-\phi(u)}=\lambda_{U}(C)=0.\end{equation}
In the next, we apply Sandwich Theorem to calculate the limit $\lim_{u\rightarrow1^{-}}[f(\phi(u),\psi(u))/(1-\phi(u))]$.
From the assumption $(\phi,\psi)\in\mathfrak{S}(C)$, the inequality (\ref{d1}) holds. Then, from Remark \ref{re:inequality}, we know that $\phi(u)\leq\psi(u)$ for all $u\in[0,1]$. Thus, from (\ref{EQTFU1}) it is easy to verify that
$$f(\phi(u),\psi(u))\leq f(\phi(u),\phi(u)).
$$
Also, notice that $f(u_{1},u_{2})\geq 0$ for all $(u_{1},u_{2})\in[0,1]^{2}$. Then, for $u\in[0, 1]$ we have
\begin{equation}\label{TFU1+++}0\leq\frac{f(\phi(u),\psi(u))}{1-\phi(u)}
\leq\frac{f(\phi(u),\phi(u))}{1-\phi(u)}.\end{equation}
From (\ref{TFU1+}), we know that when $u\rightarrow1^{-}$, the limit of the right side in the inequality (\ref{TFU1+++}) equals 0. Thus from (\ref{TFU1+++}) we get
$\lim_{u\rightarrow1^{-}}$
$f(\phi(u),\psi(u))/(1-\phi(u))=0$.
Substituting the above results into (\ref{TFU3}), we obtain
\begin{equation}\label{TFU4}\beta=\frac{1+b}{2}.\end{equation}
Then substituting (\ref{TFU4}) into (\ref{1}), we can get (\ref{up2}). \\ \\
(b). If $C(u_{1},u_{2})\equiv0$ on $[0,\varepsilon]\times [0,\varepsilon]$ for a small $\varepsilon>0$, it is easy to see that $\lambda_{L}(C)=0$.
From the assumption $\lim_{t\rightarrow0^{+}}\phi(t)/t^{\alpha}=a\in(0,+\infty)$, we can get $\phi(0)=0$ and $\phi^{[-1]}(0)=0$.
Thus from (\ref{cam2}), we have
$$\lambda_{L}(C_{\phi,\psi})=\lim_{u\rightarrow0^{+}}\frac{\phi^{[-1]}(C(\phi(u),\psi(u)))}{u}= \lim_{u\rightarrow0^{+}}\frac{\phi^{[-1]}(0)}{u}=0.$$
Then $\lambda_{L}(C_{\phi,\psi})=\left(\lambda_{L}(C)\right)^{\frac{1}{\alpha}}=0$ holds in case (1) and case (2).

Next we assume that $C(u_{1}, u_{2})>0$ for all $u_{1},u_{2}>0$. Then we have
$$\lambda_{L}(C_{\phi,\psi})=\lim_{u\rightarrow0^{+}}\frac{\phi^{[-1]}(C(\phi(u),\psi(u)))}{u}\ \ \ \ \ \ \ \ \ \ \ \ \ \ \ \ \ \ \ \ \ \ \ \ \ \ \ \ \ \ \ \ \ \ \ \ \ \ \ \ \ \ \ \ \ \ \ \ \ \ \ \ \ \ \ \ \ \ \ \ \ \ \ \
$$
\begin{equation}\label{2}\ \ \ \ \ \ \ \ \ \ \ \ \ \ \ \ =\lim_{u\rightarrow0^{+}}\left\{\left(\frac{\phi^{[-1]}(C(\phi(u),\psi(u)))}{(C(\phi(u),\psi(u)))^{\frac{1}{\alpha}}}
\cdot\frac{(\phi(u))^{\frac{1}{\alpha}}}{u}\right)\cdot\left(\frac{C(\phi(u),\psi(u))}{\phi(u)}\right)^{\frac{1}{\alpha}}\right\}.\end{equation}
 Next we calculate $\gamma:=\lim_{u\rightarrow0^{+}}C(\phi(u),\psi(u))/\phi(u)$ and
 $$\tilde{\gamma}:=\lim_{u\rightarrow0^{+}}\left(\frac{\phi^{[-1]}(C(\phi(u),\psi(u)))}{(C(\phi(u),\psi(u)))^{\frac{1}{\alpha}}}
\cdot\frac{(\phi(u))^{\frac{1}{\alpha}}}{u}\right)$$
 respectively.

 We first calculate the limit $\tilde{\gamma}$. From the assumption $\lim_{t\rightarrow0^{+}}\phi(t)/t^{\alpha}=a\in(0,+\infty)$, we know that $\lim_{t\rightarrow0^{+}}t/(\phi^{[-1]}(t))^{\alpha}=a$. Since $\phi(0)=0$, we have $C(\phi(0),\psi(0))=0$. Thus we get
\begin{equation}\label{lo+}\tilde{\gamma}=\lim_{u\rightarrow0^{+}}\left(\frac{\phi^{[-1]}(C(\phi(u),\psi(u)))}{(C(\phi(u),\psi(u)))^{\frac{1}{\alpha}}}
\cdot\frac{(\phi(u))^{\frac{1}{\alpha}}}{u}\right)=1.\end{equation}

In the next, we also apply Sandwich Theorem to calculate the limit $\gamma$. Dividing each part of the inequality (\ref{TFUinequality}) by $\phi(u)$, we get
\begin{equation}\label{3}\frac{C(\phi(u),\phi(u))}{\phi(u)}\leq\frac{C(\phi(u),\psi(u))}{\phi(u)}
\leq\frac{C(\psi(u),\psi(u))}{\phi(u)}.\end{equation} Note that from the assumption $\lim_{u\rightarrow0^{+}}\phi(u)/\psi(u)=b$, we know that $\psi(0)=0$ in the case $b\in(0,1]$. From (\ref{cam2}) we have
\begin{equation}\label{lowerequal}\lim_{u\rightarrow0^{+}}\frac{C(\phi(u),\phi(u))}{\phi(u)}=\lim_{u\rightarrow0^{+}}\frac{C(\psi(u),\psi(u))}{\psi(u)}=\lambda_{L}(C).\end{equation}
 \\
(1). We first consider the case $b=1$. Note that the limit of the left side in the inequality (\ref{3}) is given in (\ref{lowerequal}) and is $\lambda_{L}(C)$.
From (\ref{lowerequal}) and the assumption that $\lim_{u\rightarrow0^{+}}\phi(u)/\psi(u)=1$, we also derive
  $$\lim_{u\rightarrow0^{+}}\frac{C(\psi(u),\psi(u))}{\phi(u)}=\lim_{u\rightarrow0^{+}}\left(\frac{C(\psi(u),\psi(u))}{\psi(u)}\cdot\frac{\psi(u)}{\phi(u)}\right)=\lambda_{L}(C).$$
 Thus, letting $u\rightarrow0^{+}$ and applying Sandwich Theorem to the inequality (\ref{3}), we get $$\gamma=\lim_{u\rightarrow0^{+}}\frac{C(\phi(u),\psi(u))}{\phi(u)}=\lambda_{L}(C).$$ Substituting this result and (\ref{lo+}) into (\ref{2}), we can derive the formula (\ref{lo1}).
\\
(2). The result in the case $\lambda_{L}(C)=0$, $b\in(0,1]$ can be calculated similarly and we omit the proof.

The proposition is proved.
\end{proof}

\begin{example}
\label{EX: FGMTail} (Tail dependence coefficient of the TF-FGM copula) Consider the TF-FGM copula defined by (\ref{FGM}) with $\theta\in[0,1]$. Note that the FGM copula has no tail dependence, i.e., $\lambda_{L}(C)=\lambda_{U}(C)=0$.

(a). Assume that $\phi(t)=t^{\beta}$ and $\psi(t)=t^{\beta}(2-t^{\gamma})$, $t\in[0,1]$, where $0\leq\gamma\leq\beta\leq1$. It is easy to see that functions $\phi$ and $\psi$ satisfy the assumptions of Theorem \ref{thm:Distorted copula}, Proposition \ref{prop:utail}(a-2) and Proposition \ref{prop:utail}(b-2). In this case, the upper tail dependence coefficient $\lambda_{U}(C_{\phi,\psi})$ of the TF-FGM copula defined in (\ref{FGM}) equals to $\gamma/\beta$.

(b). Assume that $\phi(t)=t^{\beta}$ and $\psi(t)=t^{\gamma}$, $t\in[0,1]$, where $0\leq\gamma\leq\beta\leq1$. In this case, the assumptions of Theorem \ref{thm:Distorted copula} and Proposition \ref{prop:utail}(a-2) are also satisfied. Then from Proposition \ref{prop:utail}(a-2), it is easy to obtain that $\lambda_{U}(C_{\phi,\psi})=1-\gamma/\beta$.

If $\beta=\gamma$, the assumptions of Proposition \ref{prop:utail}(b-1) are satisfied. It is easy to show that $\lambda_{L}(C_{\phi,\psi})=0$ in this case. Although the assumptions of Proposition \ref{prop:utail}(b) are not satisfied in the case $\beta>\gamma>0$, we can get $\lambda_{L}(C_{\phi,\psi})=0$ by direct calculation.
\end{example}
\

\begin{rem}\label{re:UL}
(1) Eqs (\ref{up1})-(\ref{lo1}) show the influence of the transformation on the tail dependence coefficient of the base copula $C$. We can see that the TF copula $C_{\phi,\psi}$ can have upper tail dependence even though the base copula $C$ has no upper tail dependence.

\noindent (2) In Proposition \ref{prop:utail}(b), the results on the lower tail dependence coefficients of TF copulas, are only given in the case $b=1$ or the case $\lambda_{L}(C)=0$, $b\in(0,1]$. In such cases, the lower tail dependence coefficient is related to $\lambda_{L}(C)$ only. For the other cases, the result is complicated. For instance, choosing the FGM copula of type (\ref{FG}) as the base copula and letting $\phi(t)=t$, $\psi(t)=(1-\alpha)t+\alpha$, $\alpha\in(0,1]$, it is easy to see that the assumptions of Theorem \ref{thm:Distorted copula} are satisfied and thus $C_{\phi,\psi}$ is a copula. Also, notice that $b=\lim_{u\rightarrow0^{+}}(\phi(u)/\psi(u))=0$ for all $\alpha\in(0,1]$ and $\lambda_{L}(C)=0$. By direct calculation, we get the lower tail dependence coefficient $\lambda_{L}(C_{\phi,\psi})=\alpha(1+(1-\alpha)\theta)$ which is related to $\alpha$.
\end{rem}
\

Next we discuss the special case $C=\Pi$. The TF-product copula $\Pi_{\phi,\psi}$ has been given in (\ref{eq:dis++}) and it has multiplicative generators. In the following, we discuss its tail dependence coefficient.

Let
$\lambda:$ $[0,1]\rightarrow$ $[0,+\infty]$ be a continuous and strictly decreasing function such that $\lambda(1)=0$, $\chi:$ $[0,1]\rightarrow$ $[0,+\infty]$ be a continuous and non-increasing function satisfying that $\chi(1)=0$ and $\chi-\lambda$ is increasing in [0,1]. Furthermore, let $\exp\circ(-\lambda):$ $[0,1]\rightarrow$ $[0,1]$ be concave. Setting $\phi(t)=\exp(-\lambda(t))$ and $\psi(t)=\exp(-\chi(t))$, we have $(\phi,\psi)\in\Phi\times\Psi$, $\phi$ is concave and $\phi/\psi$ is increasing in [0,1]. Then, we get an additive form of the TF-product copula with additive generators $\lambda$ and $\chi$ presented as
\begin{equation}\label{C7}\Pi_{\lambda,\chi}(u,v)=\lambda^{[-1]}(\lambda(u\wedge v)+\chi(u\vee v)).\end{equation}
Since the product copula $\Pi$ has no tail dependence, applying Proposition \ref{prop:utail}  we can get the tail dependence coefficient of the TF-product copula.
\begin{cor}\label{cor: Archimedean tail} Let $\Pi_{\lambda,\chi}(u,v)$ be a copula of type (\ref{C7}) generated by the pair $(\lambda,\chi)$.

(1). If for some $\alpha>0$, $$\lim_{t\rightarrow1^{-}}\frac{1-\exp(-\lambda(t))}{(1-t)^{\alpha}}=a\in(0,+\infty),\ \ \lim_{t\rightarrow1^{-}}\frac{1-\exp(-\chi(t))}{1-\exp(-\lambda(t))}=b\in[0,1],$$
 then
$\lambda_{U}(\Pi_{\lambda,\chi})=2-\left(1+b\right)^{\frac{1}{\alpha}}$.

(2). If for some $\alpha>0$,
$$\lim_{t\rightarrow0^{+}}\left(t^{\alpha}\exp(\lambda(t))\right)=a\in(0,+\infty),\ \ \lim_{t\rightarrow0^{+}}(\chi(t)-\lambda(t))=b\in(-\infty,0],$$
then
$\lambda_{L}(\Pi_{\lambda,\chi})=0$.
\end{cor}


\section{Properties of TF copulas}

This section aims at investigating some properties of the TF copula $C_{\phi,\psi}$, including the totally positive of order 2 (TP$_{2}$) and the concordance order. More precisely, we study how the properties are preserved under the transformation defined by (\ref{c1}).

\subsection{The TP$_{2}$ property}

 Following the definition of \cite{Lehmann(1966)}, given intervals $P$ and $Q$ in $\mathbb{R}$, a function $A:$ $P\times Q\rightarrow\mathbb{R}$ is said to be TP$_{2}$
if, for all $u_{1}$, $u_{2}$, $v_{1}$, $v_{2}\in\mathbb{R}$ such that $u_{1}\leq u_{2}$, $v_{1}\leq v_{2}$,
 $$A(u_{1}, v_{1})A(u_{2}, v_{2})\geq A(u_{1}, v_{2})A(u_{2}, v_{1}).$$
 For the TP$_{2}$ property of the copula, please see \citet{Nelsen(2006)}.  The TP$_{2}$ has been widely studied and used, such as in the multivariate analysis, the reliability theory and the statistical decision procedure. See for example, \cite{Lai and Balakrishnan(2009)}, \cite{Olkin and Liu(2003)} and the references therein.

In this subsection, we study how the TP$_{2}$ property of the basic copula can be transformed to its TF copula. In other words, starting with a base copula $C$ satisfying the property of TP$_{2}$, we would like to investigate the conditions on the functions $\phi$ and $\psi$ such that the TF copula  $C_{\phi,\psi}$ of type (\ref{c1}) also satisfies the TP$_{2}$ property. Preliminarily, we need the following lemma.

\begin{lem}
\label{lem: Tp}\citep[Lemma 3.1]{Durante et al.(2010)} Let the function
$C:$ $[0,1]^{2}\rightarrow [0,1]$ be increasing in every variable. If $C$ satisfies the TP$_{2}$ property, then it is 2-increasing.
 \end{lem}
Then we have the following proposition.
\begin{prop}
\label{prop: Tp1} Let $C$ be a TP$_{2}$ copula and $(\phi,\psi)\in\mathfrak{S}(C)$. If $\phi^{[-1]}\circ\exp:$ $(-\infty,0]\rightarrow[0,1]$ is log-convex, then $C_{\phi,\psi}$ is also a TP$_{2}$ copula.
\end{prop}
\begin{proof}
In order to verify that $C_{\phi,\psi}$ is TP$_{2}$, we firstly prove that $\log C_{\phi,\psi}$ is 2-increasing. Suppose that the rectangle $R=[u_{1},u_{2}]\times[v_{1},v_{2}]$ is a
subset of the unit square. Define
$$b_{ij}=C(\phi(u_{i}\wedge v_{j}), \psi(u_{i}\vee v_{j})),\ i,j\in\{1,2\}.$$
Then $V_{\log C_{\phi,\psi}}(R)=\log (\phi^{[-1]}(b_{11})) +\log (\phi^{[-1]}(b_{22}))- \log (\phi^{[-1]}(b_{12}))-\log (\phi^{[-1]}(b_{21})).$
Note that each rectangle $R$ can be decomposed as the union of at most three non-intersect sub-rectangles
$R_{i}$ described in Figure \ref{fig:case}. Furthermore,
$V_{\log C_{\phi,\psi}}(R)$ is the sum of the values $V_{\log C_{\phi,\psi}}(R_{i})$. Thus, we can verify that $\log C_{\phi,\psi}$ is 2-increasing if $\log C_{\phi,\psi}$ is 2-increasing in the three cases described in Figure \ref{fig:case}. We only prove that $\log C_{\phi,\psi}$ is 2-increasing in Case 3, i.e., $u_{1}=v_{1}$, $u_{2}=v_{2}$ and $u_{1}\leq u_{2}$, since the other two cases can be proved in a similar way.

In Case 3, $b_{11}=C(\phi(u_{1}),\psi(u_{1}))$, $b_{12}=C(\phi(u_{1}),\psi(u_{2}))=b_{21}$, $b_{22}=C(\phi(u_{2}),\psi(u_{2}))$. Since $\phi$, $\psi$ are increasing in [0,1] and $C$ is increasing in every variable, it follows that
\begin{equation}\label{eq:tp3}\log b_{11}\leq \log(\min(b_{12}, b_{21}))\leq \log(\max(b_{12}, b_{21}))\leq \log b_{22}.\end{equation}
From the assumption that $C$ is TP$_{2}$, we have $b_{11}b_{22}\geq C(\phi(u_{1}),\psi(u_{2}))\cdot C(\phi(u_{2}),\psi(u_{1}))$. Since $(\phi,\psi)\in\mathfrak{S}(C)$, then the inequality (\ref{d1}) holds such that
$C(\phi(u_{1}),\psi(u_{2}))\cdot C(\phi(u_{2}),\psi(u_{1}))
\geq C(\phi(u_{1}),\psi(u_{2}))\cdot C(\phi(u_{1}),\psi(u_{2}))=b_{12} b_{21}$. Combining the above inequalities, we know $b_{11}b_{22}\geq b_{12} b_{21}$ and applying the function $\log$ to this inequality, we get
\begin{equation}\label{eq:tp1}\log b_{11}+\log b_{22}\geq \log b_{12}+\log b_{21}.\end{equation}
Also, from the assumption that $\phi^{[-1]}\circ\exp:$ $(-\infty,0]\rightarrow[0,1]$ is log-convex, we know $\log(\phi^{[-1]}(e^{t}))$ is convex in $(-\infty,0]$. Therefore, from inequalities (\ref{eq:tp3}) and (\ref{eq:tp1}) and applying Lemma \ref{lem: MO} to the function $\log(\phi^{[-1]}(e^{t}))$, we obtain
\begin{equation}\label{eq:tp1++}V_{\log C_{\phi,\psi}}(R)=\log (\phi^{[-1]}(b_{11})) +\log (\phi^{[-1]}(b_{22}))- \log (\phi^{[-1]}(b_{12}))-\log (\phi^{[-1]}(b_{21}))\geq 0.\end{equation}
This shows that $\log C_{\phi,\psi}$ is 2-increasing.
From the inequality (\ref{eq:tp1++}), we also get
$$\phi^{[-1]}(b_{11})\phi^{[-1]}(b_{22})\geq\phi^{[-1]}(b_{12})\phi^{[-1]}(b_{21})$$
for arbitrary rectangle $R$ contained in the unit square. Hence, $C_{\phi,\psi}$ is a TP$_{2}$ function.

Now we prove $C_{\phi,\psi}$ is a copula. Form Lemma \ref{lem: Tp}, it is obvious that $C_{\phi,\psi}$ is also 2-increasing. Furthermore, we have proved that $C_{\phi,\psi}(u,v)$ satisfies the boundary conditions in the proof of Theorem \ref{thm:Distorted copula}. Thus, $C_{\phi,\psi}$ is a TP$_{2}$ copula.
\end{proof}

By Theorem \ref{thm:Distorted copula} and  Proposition \ref{prop: Tp1}, we also have the following results about the TP$_{2}$ property of $C_{\phi,\psi}$.
\begin{cor}\label{cor: power TP2} Let $(\phi,\psi)\in\Phi\times\Psi$ and $C$ be a TP$_{2}$ copula.

\noindent (1). If $C$ is supermigrative, $\phi/\psi$ is increasing in [0,1] and $\phi^{[-1]}\circ\exp:$ $(-\infty,0]\rightarrow[0,1]$ is log-convex, then $C_{\phi,\psi}$ is a TP$_{2}$ copula.

\noindent (2). Let $\phi(t)=t^{\alpha}$ for $\alpha>0$ . If the assumption (\ref{d1}) holds, then $C_{\phi,\psi}$ is a TP$_{2}$ copula.

\end{cor}
\begin{proof}
From the definition (\ref{d3}) of the supermigrative copula, using Theorem \ref{thm:Distorted copula} and Proposition \ref{prop: Tp1}, we get the first part of the corollary.

Note that every power function $\phi=t^{\alpha}$, $\alpha>0$, has the properties that $\phi^{[-1]}$ is also a power function and $\phi^{[-1]}\circ\exp$ is log-convex. Then, we get the second part of the corollary.
\end{proof}

These results allow us to construct the TP$_{2}$ multiple parameters' copula family from a known TP$_{2}$ copula. In the following, we give two examples discussed in the previous sections.
\begin{example}
\label{EX: FGMTP2} Choose the FGM copula as the base copula $C$. Then $C$ is TP$_{2}$ and supermigrative when $\theta\in[0,1]$.
Assume that $\phi(t)=t^{\beta}$ and $\psi(t)=t^{\gamma}$, where $0\leq \gamma\leq\beta\leq1$. From Corollary \ref{cor: power TP2}, we know that the TF-FGM copula defined in (\ref{FGMPower})
 is TP$_{2}$.
\end{example}
\

\begin{example}
\label{EX: CATP2} Let the base copula $C$ be a member of the Cuadras-Aug\'{e} family. Then $C$ is TP$_{2}$ for all $0\leq\alpha\leq1$. Assume that $\phi(t)=t^{\beta}$ and $\psi(t)=t^{\beta}(2-t^{\gamma})$ for all $t\in[0,1]$, where $0\leq\gamma\leq\beta\leq1$. Then, the TF copula of type (\ref{capower}) also defines a family of TP$_{2}$ copulas.
\end{example}

\subsection{Concordance order }

  We first give the definition of the concordance order between two copulas. Given two copulas $C$ and $D$, if $C(u,v)\leq D(u,v)$ for any $u,v\in[0,1]$, then $D$ is said to more concordant than $C$, denoted by $C\prec D$ \citep[see][for more details]{Joe(1997)}. In the following proposition, we discuss the concordance order of TF copulas.
\begin{prop}
\label{prop: co} (1). Let $C$ and $D$ be two bivariate copulas, $(\phi,\psi)\in\mathfrak{S}(C)\bigcap\mathfrak{S}(D)$ and $\phi$ be concave in [0,1]. If $C\prec D$, then $C_{\phi,\psi}\prec D_{\phi,\psi}$. Conversely, if $C_{\phi,\psi}\prec D_{\phi,\psi}$, then $C(u,v)\leq D(u,v)$ on $\{(u,v)\in[\phi(0),1]\times[\psi(0),1]$| $u\leq v$, $C(u,v)\geq \phi(0)\}$.

(2). Let $C$ be a bivariate copula, $(\phi,\psi_{1}), (\phi,\psi_{2})\in\mathfrak{S}(C)$ and $\phi$ be concave in [0,1]. Then $C_{\phi,\psi_{1}}\prec C_{\phi,\psi_{2}}$ if and only if $\psi_{1}(u)\leq\psi_{2}(u)$ for all $u\in[0,1]$.

(3). Let $C$ be a bivariate copula, $(\phi_{1},\psi), (\phi_{2},\psi)\in\mathfrak{S}(C)$ and $\phi_{1}$, $\phi_{2}$ be concave in [0,1]. Then $C_{\phi_{1},\psi}\prec C_{\phi_{2},\psi}$ if and only if $C(f(s),t)\leq f(C(s,t))$ for all $0\leq s\leq t\leq1$, where $f=\phi_{1}\circ\phi^{[-1]}_{2}$.
\end{prop}
\begin{proof}
We only prove (3). The proofs of (1) and (2) are similar.

From the definition of $C_{\phi,\psi}$ given in (\ref{c1}), we can know that $C_{\phi_{1},\psi}\prec C_{\phi_{2},\psi}$ if and only if
\begin{equation}
\label{eq: co}\phi_{1}^{[-1]}(C(\phi_{1}(u\wedge v),\psi(u\vee v)))\leq \phi_{2}^{[-1]}(C(\phi_{2}(u\wedge v),\psi(u\vee v))),\ u,v\in[0,1].\end{equation}

Let $s=\phi_{2}(u\wedge v)$ and $t=\psi(u\vee v)$. From the assumption $(\phi_{2},\psi)\in\mathfrak{S}(C)$, we know $C(\phi_{2}(u),\psi(v))\leq C(\phi_{2}(v),\psi(u))$ for all $u\leq v$, $(u,v)\in[0,1]^{2}$. Then, from Remark \ref{re:inequality}, we have $\phi_{2}(u)\leq \psi(u)$, $u\in[0,1]$, such that
$\phi_{2}(0)\leq s \leq t\leq1$. Hence, the inequality (\ref{eq: co}) is equivalent to that $\phi_{1}^{[-1]}(C(f(s),t))\leq \phi_{2}^{[-1]}(C(s,t))$
for all $\phi_{2}(0)\leq s \leq t\leq1$. Furthermore, if $s<\phi_{2}(0)$, then $\phi_{1}^{[-1]}(C(f(s),t))=\phi_{2}^{[-1]}(C(s,t))=0$. Therefore, $C_{\phi_{1},\psi}\prec C_{\phi_{2},\psi}$ if and only if \begin{equation}
\label{eq: co1}\phi_{1}^{[-1]}(C(f(s),t))\leq\phi_{2}^{[-1]}(C(s,t)),\ 0\leq s\leq t\leq 1.\end{equation}

We start by proving the necessity. Assume that $C_{\phi_{1},\psi}\prec C_{\phi_{2},\psi}$. Then (\ref{eq: co1}) follows. Applying $\phi_{1}$ to the both sides of (\ref{eq: co1}) and noting that $\phi_{1}$ is non-decreasing in [0,1] as well as $\phi_{1}(\phi_{1}^{[-1]}(t))\geq t$ for all $t\in[0,1]$, it yields
that $$C(f(s),t)\leq\phi_{1}\circ\phi_{1}^{[-1]}(C(f(s),t))\leq\phi_{1}\circ\phi_{2}^{[-1]}(C(s,t))=f(C(s,t))$$
for all $0\leq s\leq t\leq 1$.

Next we prove the sufficiency. If $f$ and $C$ satisfy $C(f(s),t)\leq f(C(s,t))$, $0\leq s\leq t\leq 1$, then applying $\phi_{1}^{[-1]}$ to the both sides of this inequality and noting that $\phi_{1}^{[-1]}$ is non-decreasing in $[0,1]$ and $\phi_{1}^{[-1]}\circ f=\phi_{2}^{[-1]}$, it yields that (\ref{eq: co1}) holds, i.e., $C_{\phi_{1},\psi}\prec C_{\phi_{2},\psi}$.

This completes the proof.
\end{proof}

\section{\label{sec:Conclusion}Conclusions}

In this paper, we presented a family of bivariate copulas by transforming a given copula function with two increasing functions, named as transformed copula. Conditions guaranteeing that the transformed function is a copula function are provided, and several classes of transformed copulas are given. A singular component along the main diagonal of the transformed copula is verified, and the tail dependence coefficients of the transformed copulas are obtained. The transformed copula preserve some orders and properties of the base copula $C$, such as the TP$_{2}$ property and the concordance order.

\subsection*{Acknowledgement}
Xie's research was supported by the National Natural Science Foundation of China
(Grants No. 11561047). Yang's research was supported by the National Natural Science Foundation of China
(Grants No. 11671021, Grants No. 11271033).

\appendix

\section{Proof of Remark \ref{REM1}}

When $C$ is exchangeable, $C_{2}(u,v)=\frac{\partial C(u,v)}{\partial v}=\frac{\partial C(v,u)}{\partial v}=C_{1}(v,u)=g(v,u)h(v)$. Also, if $g(u,v)$ is exchangeable, then $C_{2}(u,v)=g(u,v)h(v)$. Furthermore, if $\int_{\psi(u)}^{\phi(u)}h(t)dt$ is strictly increasing in $[0,1]$, we know that $h(\phi(u))\phi^{'}(u)-h(\psi(u))\psi^{'}(u)>0$ for all $u\in[0,1]$.

    Summarizing the above results, we conclude that for all $u\in[0,1]$,
     $$C_{1}(\phi(u),\psi(u))\phi'(u)=g(\phi(u),\psi(u))h(\phi(u))\phi'(u)>g(\phi(u),\psi(u))h(\psi(u))\psi'(u)=C_{2}(\phi(u),\psi(u))\psi'(u),$$
    i.e., $S\neq\emptyset$. From Theorem \ref{thm:singular}, we know that $C_{\phi,\psi}$ contains a singular components along the main diagonal.

\end{document}